\documentclass[a4paper, 11pt]{amsart}
\usepackage[utf8]{inputenc}
\usepackage{microtype}
\usepackage{amsmath}
\usepackage{amssymb}
\usepackage{amsthm}
\usepackage{amsfonts}
\usepackage{booktabs}
\usepackage[final]{graphicx}
\usepackage{setspace}
\usepackage[center]{caption}
\usepackage{tikz}
\usetikzlibrary{shapes}
\usepackage{tikz-cd}
\usepackage[active]{srcltx}
\usepackage{mathtools}
\usepackage{mathrsfs}
\usepackage{ytableau}
\usepackage{mleftright}
\usepackage{comment}
\usepackage[colorlinks=true, allcolors=blue]{hyperref}
\usepackage[capitalise]{cleveref}
\usepackage[top=3cm,bottom=2cm,left=3cm,right=3cm,marginparwidth=1.75cm]{geometry}
\usepackage{algorithm}
\usepackage{algorithmic}

\usepackage{xltabular}

\usepackage[colorinlistoftodos]{todonotes}

\usepackage[shortlabels]{enumitem}
\setlist[1]{itemsep=2pt}
\usepackage[capitalise]{cleveref}

\usepackage{autonum} 

\usepackage{bm}
\usepackage{dsfont}
\usepackage[mathcal]{euscript}

\usepackage{enumitem}
\usepackage{relsize}
\theoremstyle{plain}\newtheorem{theorem}{Theorem}[section]
\theoremstyle{plain}\newtheorem{proposition}[theorem]{Proposition}
\theoremstyle{plain}\newtheorem{lemma}[theorem]{Lemma}
\theoremstyle{plain}
\theoremstyle{definition}\newtheorem{definition}{Definition}
\theoremstyle{definition}\newtheorem{remark}[]{Remark}
\theoremstyle{definition}
\theoremstyle{definition}
\theoremstyle{definition}
\theoremstyle{definition}

\theoremstyle{thm}

\newcommand{\HH}{\mathcal{H}}

\newcommand{\RR}{\mathcal{R}}
\newcommand{\Pol}{\R[x]_{\leq d}}

\newcommand{\R}{\mathbb{R}}
\newcommand{\C}{\mathbb{C}}
\newcommand{\N}{\mathbb{N}}

\newcommand{\n}[1]{\left\lVert#1\right\rVert}

\newcommand{\dist}{\mathrm{dist}}

\title{Testing the variety hypothesis}
\author{A. Lerario, P. Roos Hoefgeest, M. Scolamiero, and A. Tamai}
\begin{document}
\maketitle
\begin{abstract}

Given a probability measure on the unit disk, we study the problem of deciding whether, for some threshold probability, this measure is supported near a real algebraic variety of given dimension and bounded degree. We call this ``testing the variety hypothesis''.
We prove an upper bound on the so--called ``sample complexity'' of this problem and show how it can be reduced to a semialgebraic decision problem. This is done by studying in a quantitative way the Hausdorff geometry of the space of real algebraic varieties of a given dimension and  degree.
\end{abstract}

\section{Introduction}
The manifold hypothesis is a foundational idea in machine learning, suggesting that high--dimensional data often concentrate near a low--dimensional smooth manifold. This concept has been formalized in recent works such as \cite{Fefferman, Narayanan}, which introduce algorithms for testing whether a probability measure is supported near a manifold with controlled geometry—typically quantified via bounds on volume and reach.

In this paper, we propose an algebraic analogue of that framework, which we call ``testing the variety hypothesis''. Given a probability measure on the unit disk, we ask whether it is mostly concentrated near a real algebraic variety of fixed dimension and bounded degree. This reformulation broadens the scope of the original problem by allowing for singular supports, which can be more appropriate when smoothness assumptions are too restrictive. 

The following theorem is a conceptual consequence of our main probabilistic result (\cref{thm:Nvariationintro}) and the semialgebraic reformulation developed in \cref{sec:decision}. It asserts the existence of an algorithm that, with high confidence, decides whether a given distribution is well--approximated by a real algebraic variety of bounded complexity. While we do not construct or analyze such an algorithm explicitly, the reduction to a semialgebraic decision problem guarantees its existence within the standard framework of real algebraic geometry.

\begin{theorem}[Testing the variety hypothesis]\label{thm:decision}
Let \( \mu \) be a Borel probability measure supported on \( D^n \). Given a target dimension \( k \), a degree bound \( d \), an error tolerance \( \epsilon > 0 \), and a confidence level \( \delta \in (0,1) \), there exists an algorithm that—after observing at least  
\[
m \geq \frac{1}{\epsilon^2} \log\left( \frac{c_1}{\epsilon^{c_2} \delta} \right)
\]  
independent samples from \( \mu \)—will, with probability at least \( 1 - \delta \), distinguish between the following two situations (at least one must hold):
\begin{enumerate}
  \item There exists a real algebraic variety \( Z \subset \R^n \) of dimension \( k \) and degree at most \( 2d \) such that  
  \[
  \int_{D^n} \mathrm{dist}(x, Z)^2 \, d\mu(x) < \tfrac{3}{2} \epsilon,
  \]
  \item or, for every real algebraic variety \( Z \subset \R^n \) of dimension \( k \) and degree at most \( d \), it holds that  
  \[
  \int_{D^n} \mathrm{dist}(x, Z)^2 \, d\mu(x) > \tfrac{1}{2} \epsilon.
  \]
\end{enumerate}
Here, \( \mathrm{dist}(x, Z) \) denotes the Euclidean distance from the point \( x \) to the set \( Z \), and $c_1, c_2 > 0$ depend only on $n,k,d$.
\end{theorem}

The main contribution of this work lies in developing the conceptual and mathematical framework underpinning this result. It combines ideas from statistical learning theory—such as the notion of sample complexity—with tools from real algebraic geometry that allow the extension of learnability results to noncompact and stratified hypothesis classes. In particular, we develop a quantitative study of the Hausdorff geometry of the space of real algebraic varieties, which plays a crucial role in handling complexity estimates and is of independent mathematical interest. We use the remainder of the introduction to explain the key ideas behind our approach and the main results of the paper.

\subsection{A notion of controlled geometry}
The varieties considered throughout this paper—those referenced in the title and central to our testing problem—are real algebraic sets with \emph{controlled geometry}, meaning they have fixed dimension and are defined by polynomial equations of bounded degree.

This choice is motivated by the need to generalize the manifold hypothesis, commonly assumed in nonlinear dimensionality reduction, to a broader class of spaces. While the manifold hypothesis postulates that data lie near a low--dimensional smooth manifold, real--world datasets often violate smoothness assumptions and include singularities or stratified structures. To model such settings, we require a class of geometric objects that offers both flexibility and mathematical tractability.

Real algebraic sets provide a natural solution. They encompass singular structures, admit finite algebraic descriptions, and lend themselves to semialgebraic formulations of optimization and learning problems. This makes them an ideal candidate for extending manifold--based approaches to more general scenarios.

Formally, for polynomials $p_1, \ldots, p_s \in \R[x_1, \ldots, x_n]$, we define their common zero set as
\[
Z(p_1, \ldots, p_s) := \left\{ x \in \R^n \,\bigg|\, p_1(x) = \cdots = p_s(x) = 0 \right\}.
\]

If each polynomial $p_j$ has degree at most $d$, we say that the corresponding real algebraic set $Z$ (which we will also refer to as a “variety”) has degree at most $d$, and write $\deg(Z) \leq d$. We do not impose irreducibility or define degree via a minimal representation, as our interest is in bounding complexity rather than enforcing uniqueness.

Let $D^n \subset \R^n$ denote the closed unit disk. For any variety $Z \subset \R^n$, we define the dimension of $Z \cap D^n$, denoted $\dim(Z \cap D^n)$, to be its topological dimension as a semialgebraic set. Note that while $Z$ is algebraic, its intersection with $D^n$ is generally semialgebraic rather than algebraic. For more details on these notions, we refer the reader to \cref{sec:semialgebraic}.

We now define our \emph{hypothesis class}, the key geometric class used throughout the paper:
\[
\HH(n,k,d) := \left\{ Z \cap D^n \,\bigg|\, \dim(Z \cap D^n) = k, \;\deg(Z) \leq d  \right\}.
\]
Elements of $\HH(n,k,d)$ are not required to be smooth, connected, or even manifolds. Nevertheless, their geometry remains controlled in a strong sense: by Thom--Milnor's bound, the number of their connected components is $O(d^n)$ \cite{Milnor}, and their $k$--dimensional Hausdorff volume is bounded by $O(d^k)$, via classical results in integral geometry \cite{ComteYomdin}.  Given a probability measure, can we test whether it concentrates near one of these varieties?

\subsection{Mathematical formulation of the problem}
To formulate the problem rigorously, let now $\mu$ be a Borel probability measure supported on $D^n$. The test for the variety hypothesis answers the following affirmatively: given error $\epsilon>0$, dimension $k$, degree $d$ and confidence $1-\delta$, is there an algorithm that takes a number, depending on these parameters, of samples from $\mu$ and with probability at least $1-\delta$ distinguishes between the following two cases (as least one must
hold):
\begin{enumerate}
\item there is $Z\in \HH(n, k, 2d)$ such that 
\[ \int_{D^n}\mathrm{dist}(x, Z)^2\mu(\mathrm{d}x)< \tfrac{3}{2}\epsilon.\]
\item for all $Z\in \HH(n, k, d)$
\[\int_{D^n}\mathrm{dist}(x, Z)^2\mu(\mathrm{d}x)> \tfrac{1}{2}\epsilon.\]
\end{enumerate}

The basic statistical question that underlies this problem is: what is the minimal number of samples (the so--called ``sample complexity'') needed for testing the hypothesis that data, which are supposed to be distributed according to the unknown measure $\mu$, lie near an algebraic variety with bounded dimension and degree?

Quantitatively, this is studied by introducing the \emph{risk} functional $\RR(\cdot, \mu):\mathcal{H}\to \R$, 
\[\RR(Z, \mu):=\int_{D^n}\mathrm{dist}(x, Z)^2\mu(\mathrm{d}x),\]
and the abstract goal is to decide if the infimum of $\RR(\cdot, \mu)$ on our hypothesis class is below a certain threshold. This  problem can be approached following the conceptual strategy of \cite{Fefferman} as follows. The process of sampling $m$ independent points\footnote{We use the following convention: a point in $\mathbb{R}^n$ is denoted by $x = (x_1, \ldots, x_n)$, where $x_i$ indicates the $i$-th coordinate. Superscripts are used to index samples, so $x^i$ denotes the $i$-th sample.} $x^1, \ldots, x^m$ from $\mu$ defines a random measure
\[\mu_m:=\frac{1}{m}\sum_{i=1}^m \delta_{x^i},\]
called the random empirical measure. The \emph{empirical risk} is then defined as the value of the risk with respect to the empirical measure:
\[\RR(Z, \mu_m)=\frac{1}{m}\sum_{i=1}^m\mathrm{dist}(x^i, Z)^2.\]
Note that, once $\mu$ is fixed, the empirical risk is a random functional on the class $\mathcal{H}(n,k,d)$ (the randomness coming from $\mu$ being a probability measure). 

The next theorem, which corresponds to \cref{thm:Nvariation} below, is one of our main results. It reduces the above problem to deciding if the infimum of the empirical risk is below a certain threshold, as long as the number of samples is large enough.


\begin{theorem}\label{thm:Nvariationintro}For every $n,k,d\in \N$ there exist $c_0,c_1, c_2, c_3>0$ such that the following statement is true. Let $\mu$ be a Borel probability measure on $D^n$ and for all $m\in \N$ denote by $\mu_m$ the corresponding random empirical measure. For all $0<\epsilon< c_0$ and $0<\delta<1$, 
if
\[\label{eq:boundnkdintro}m\geq \frac{c_3}{\epsilon^2}\log\left(\frac{c_1 \epsilon^{-c_2}}{\delta}\right),\]
then
\[\mathbb{P}\left\{\sup_{Z\in \HH(n,k,d)}\left|\RR(Z, \mu_m)-\RR(Z,\mu)\right|> \epsilon\right\}<\delta.\]
\end{theorem}
The main idea for this result is to pair the concentration bound from \cref{lemma:balls} (essentially Hoeffding's inequality) with an $\epsilon$--covering argument for the space $\HH(n,k,d)$ with respect to the Hausdorff metric: since $\HH(n,k,d)$ is precompact in the space of compact subsets of the disk with this metric, such a finite covering always exists, but the key point is controlling its cardinality; roughly speaking it should grow in a controlled way as $\epsilon\to 0$, which is the content of \cref{coro:netintro}, discussed below.
\subsection{Limitations of Manifold Learning in the algebraic setting}\label{sec:introexisting}

In principle, existing results from the manifold learning literature—such as \cite[Theorem 4]{Narayanan}—can be applied to the class $\HH(n,k,d)$ of algebraic sets. These results provide sample complexity bounds for empirical risk minimization (ERM) over families of sets that admit covering number estimates. However, while promising, this black-box application has intrinsic limitations when extended to singular or more general algebraic structures.

To illustrate how far existing techniques can go, consider the following: from \cite[Theorem 2.6]{KileelZhang}, we know that the number of $\epsilon$--balls needed to cover an algebraic set $Z \in \HH(n,k,d)$ satisfies the bound
\[
N(\epsilon, Z) \leq \left(\frac{4kn^{3/2}}{\epsilon}\right)^k (4d)^{n-k}.
\]
This gives an estimate for the covering number of the entire class $\HH(n,k,d)$ and can be plugged into \cite[Theorem 4]{Narayanan} to obtain an explicit bound on the number of samples required for empirical risk minimization. The resulting sample complexity satisfies:
\[
m \le C\left(\frac{(kn^{3/2})^k d^{n-k}}{4^{2k-n} \epsilon^{k+2}} \min\left\{ \frac{(kn^{3/2})^k d^{n-k}}{4^{2k-n} \epsilon^k}, \frac{1}{\epsilon^2} \log^4\left( \frac{(kn^{3/2})^k d^{n-k}}{4^{2k-n} \epsilon^{k+1}} \right) \right\} + \frac{1}{\epsilon^2} \log\frac{1}{\delta} \right),
\]
with $C>0$ some explicit constant.
This bound has the advantage of being fully explicit in terms of the parameters $(n, k, d)$, but deteriorates rapidly as $\epsilon \to 0$, with a worst--case dependence of order $O(\epsilon^{-k-4})$.

Moreover, existing methods such as those in \cite{Narayanan, Fefferman} rely heavily on geometric regularity. Their algorithm is designed for classes like
\[
\HH_{\mathrm{sm}}(k, V, \rho) := \left\{ Z \subset D^n \,\middle|\, \dim(Z) = k,\ \mathrm{vol}(Z) \leq V,\ \mathrm{reach}(Z) \geq \rho \right\},
\]
where uniform control over geometry—via volume and reach—ensures compactness of the class itself and bounded curvature for its elements. However, such assumptions break down in the algebraic setting: even smooth algebraic sets of fixed degree can have arbitrarily small reach, and the family of all such sets is not compact in whatever reasonable topology. We refer the reader to \cref{sec:sample} below for a more detailed discussion.

Unlike previous approaches, our framework avoids any regularity assumptions on curvature, volume, or reach. \cref{thm:Nvariationintro} shows that uniform convergence of empirical risk holds over $\HH(n,k,d)$, with sample complexity scaling as $O(\epsilon^{-2} \log(\epsilon^{-1}))$, significantly improving over standard covering number bounds. This is made possible by a quantitative analysis of the Hausdorff geometry of real algebraic sets, which we develop in the next section.

\subsection{Hausdorff geometry of complete intersections}
The main technical tools underlying our study rely on the geometry of real complete intersections, for which we prove results of independent interest. They belong to the emerging field of \emph{metric algebraic geometry}, which aims to join methods and ideas of (metric) differential geometry with its (real) algebraic counterpart; see \cite{metricalgebraicgeometry, lecturenotesantonio}. We elaborate these results here.

Consider the compact metric space $\mathcal{K}(D^n)$ consisting of compact subsets of the unit disk, equipped with the Hausdorff distance, denoted by $\operatorname{dist}_H$. Let $\Pol$ denote the vector space of all $c$--tuples $p=(p_1, \ldots, p_c)$ of real polynomials in $n$ variables, each polynomial having degree at most $d$. This space has dimension $c\binom{n+d}{d}$, corresponding to the number of polynomial coefficients. Every $p \in \Pol$ defines a real algebraic set:
\[
Z(p)=\left\{x \in \R^n  \,\bigg|\, p_1(x)=\cdots=p_c(x)=0\right\}.
\]
Since some polynomial systems may have no solutions in  $D^n$, we introduce the set
\[
P(d):=\left\{p \in \Pol \,\bigg|\, Z(p) \cap D^n \neq \emptyset\right\},
\]
and define the map
\[
\kappa: P(d) \to \mathcal{K}(D^n), \quad p \mapsto Z(p) \cap D^n,
\]
which associates to each polynomial tuple its zero set on the disk. Note that this map is generally not continuous, since solutions may disappear under small perturbations of polynomial coefficients (see \cref{lemma:continuity}).

We also introduce the \emph{discriminant} $\Sigma(d) \subseteq P(d)$, consisting of polynomial tuples for which the equation $p=0$ fails to be regular on $D^n$ (interpreted as a manifold with boundary, see \cref{def:discriminant}). The discriminant is a closed subset, and its complement
\[
T(d):=P(d) \setminus \Sigma(d)
\]
consists exactly of polynomials describing smooth, nonempty complete intersections\footnote{The reader can take this as definition.} in $D^n$. Hence, for $p \in T(d)$, the set $Z(p)$ is a smooth manifold of dimension $(n-c)$, possibly with a smooth boundary on $\partial D^n$.

Although $\kappa$ is not globally continuous, we show that it is locally Lipschitz on $T(d)$, with a Lipschitz constant inversely proportional to the distance from the discriminant (see \cref{thm:thom} in the body of the paper).

\begin{theorem}\label{thm:thomintro}
For every $n,k,d \in \N$, there exists $L>0$ such that for every $p \in T(d)$ and sufficiently small $\|p-q\|$,

$$
\mathrm{dist}_H\left(\kappa(p), \kappa(q)\right)\leq \frac{L}{\mathrm{dist}(p, \Sigma(d))}\cdot \|p-q\|.
$$

\end{theorem}

Here the norm $\|\cdot\|$ on the space of polynomials is the Bombieri--Weyl norm (see \cref{def:BW}), and $\operatorname{dist}$ refers to the induced distance. This result is proved following in a quantitative way the proof of the classical Thom's Isotopy Lemma, detailed in \cref{propo:thom}.

(Naturally, being away from the discriminant controls extrinsic geometric properties. For instance, we prove in \cref{propo:reach} that the reach of a smooth complete intersection is bounded from below by a constant multiple of its distance to the discriminant.)

Near the discriminant, the behavior of $\kappa$ is more delicate. Previously, Basu and Lerario \cite[Theorem 2.10]{BasuLerario} proved that the Hausdorff closure of complete intersections of degree $2d$ contains the discriminant in degree $d$:
\[
\kappa(\Sigma(d)) \subseteq \overline{\kappa(T(2d))}^{H}.
\]
We strengthen this result by establishing a quantitative version (\cref{thm:tau} below).

\begin{theorem}\label{thm:tauintro}
For each $n,k,d \in \mathbb{N}$, there exist $\alpha, \beta, \epsilon_0>0$ such that, for every $0<\epsilon<\epsilon_0$ and every $Z \in \HH(n,k,d)$ with $\dim(Z)=k=n-c$, one can find a polynomial tuple $p_\epsilon \in P(2d)$ satisfying:

$$
\mathrm{dist}_H(Z, Z(p_\epsilon)\cap D^n)\leq \epsilon\quad \text{and}\quad \mathrm{dist}(p_\epsilon,\Sigma(2d))\geq \|p_\epsilon\|\alpha \epsilon^\beta.
$$

\end{theorem}

The quantitative nature of this theorem reflects the definability of Hausdorff limits within semialgebraic families, a classical result we invoke in its formulation from \cite{definability}.

Combining \cref{thm:thomintro} and \cref{thm:tauintro}, we achieve explicit control over the compactness properties of $\HH(n,k,d)$ within $\mathcal{K}(D^n)$ (see \cref{coro:net} below).

\begin{theorem}\label{coro:netintro}
For every $n,k,d$, there exist $\epsilon_0,a_1,a_2>0$ such that for all $0<\epsilon<\epsilon_0$, there exist polynomials $q_1, \ldots, q_{\nu(\epsilon)}\in P(2d)\setminus \Sigma(2d)$ satisfying
$$
\HH(n,k,d) \subseteq \bigcup_{i=1}^{\nu(\epsilon)}B_H(Z(q_i)\cap D^n, \epsilon),\quad \text{with}\quad \nu(\epsilon)\leq a_1\epsilon^{-a_2}.
$$
\end{theorem}

The results in this section provide the geometric foundation for our main statistical estimate, establishing both local stability and global approximability of smooth complete intersections under the Hausdorff metric. Together, they show that the class $\HH(n,k,d)$ admits uniform approximation by smooth varieties of bounded degree and controlled distance from the discriminant. This quantitative precompactness underlies the covering estimates used in the proof of \cref{thm:Nvariationintro} and demonstrates how algebraic structure can substitute for geometric regularity. Beyond its theoretical role, \cref{coro:netintro} also suggests a potential algorithmic strategy for testing the variety hypothesis via finite approximation. We now turn to a semialgebraic reformulation of the problem, which connects this geometric structure with decision procedures in real algebraic geometry.

\subsection{From empirical risk minimization to a semialgebraic decision problem: proof sketch of \cref{thm:decision}}\label{sec:decision}

We now explain how the probabilistic estimate in \cref{thm:Nvariationintro}, together with the geometric covering results of the previous section, leads to a decision--theoretic reformulation of the variety hypothesis problem. Specifically, we show that testing whether a distribution is close to being supported on a real algebraic variety reduces to checking the nonemptiness of a semialgebraic set defined by constraints on degree, dimension, smoothness, and empirical proximity to a finite dataset. This reformulation underlies \cref{thm:decision}, and illustrates the connection between statistical learning and real algebraic geometry.

\subsubsection{Semialgebraic decision formulation.}
The variety hypothesis problem asks whether the support of an unknown distribution $\mu$ can be approximated by a real algebraic variety of dimension $k$ and degree at most $d$. Given an accuracy parameter $\epsilon > 0$, this corresponds to testing whether
\[
\inf_{Z \in \HH(n,k,d)} \RR(Z, \mu) \leq \epsilon.
\]
Since $\mu$ is unknown, we replace the true risk with the empirical risk $\RR(Z, \mu_m)$ computed on $m$ samples $x_1, \ldots, x_m \in D^n$. By \cref{thm:Nvariationintro}, this yields a reliable estimate with high probability, provided $m$ is large enough. Therefore, we reduce to testing whether
\[
\inf_{Z \in \HH(n,k,d)} \RR(Z, \mu_m) \leq \epsilon.
\]

Each candidate variety \( Z \in \mathcal{H}(n,k,d) \) is defined as the zero set of a tuple of real polynomials \( p = (p_1, \ldots, p_c) \), each of degree at most \( d \). Conditions such as the nonemptiness of \( Z(p) \cap D^n \), its dimension being \( k \), and the empirical risk bound
\[
\mathcal{R}(Z(p), \mu_m) = \frac{1}{m} \sum_{i=1}^m \mathrm{dist}(x^i, Z(p))^2 \leq \epsilon
\]
are all \emph{definable by first-order formulas over the reals}. These conditions involve polynomial inequalities in the coefficients of \( p \), typically with several quantifiers. Consequently, the existence of a variety \( Z(p) \) satisfying these constraints corresponds to the satisfiability of a first-order formula.

Specifically, letting \( N := \dim(\Pol) + 1 \), observe firts that we may restrict\footnote{The requirement that \( p = p_1^2 + \cdots + p_s^2 \) means \( p \in \mathcal{S}_{2d} \), the cone of sums of squares of degree \(\leq 2d\). Since \(\mathcal{S}_{2d}\) is a convex cone in \(\Pol\), Carathéodory’s theorem implies that any such \( p \) can be written as a sum of at most \( \dim(\Pol) + 1 \) squares.} to polynomials of the form \( p = p_1^2 + \cdots + p_N^2 \), where \( \deg(p_i) \leq d \). Let also $\underline{x}_m:=(x^1, \ldots, x^m)\in (D^n)^m$. The problem then reduces to deciding the nonemptiness of the semialgebraic set
\[
W(\underline{x}_m, \epsilon) :=
\left\{ p \in \mathbb{R}[x_1, \ldots, x_n]_{\leq 2d} \,\middle|\,
\begin{array}{l}
p = p_1^2+\cdots+p_N^2, \\[2pt]
\dim(Z(p) \cap D^n) = k, \\[2pt]
\frac{1}{m} \sum_{i=1}^m \mathrm{dist}(x^i, Z(p) \cap D^n)^2 \leq \epsilon
\end{array}
\right\},
\]
which is semialgebraic because it is defined by a first-order formula over the reals.

To determine whether \( W(\underline{x}_m, \epsilon) \) is nonempty, one can apply effective quantifier elimination. Define the semialgebraic set
\[
V := \left\{ (\underline{x}_m, \epsilon) \in (D^n)^m \times \mathbb{R} \,\bigg|\, \exists p \in W(\underline{x}_m, \epsilon) \right\}.
\]
Effective quantifier elimination (see ~\cite[Chapter 14]{BasuPollackRoy}) guarantees that \( V \) can be expressed as a finite Boolean combination of polynomial sign conditions:
\[
V = \bigcup_{a=1}^A \bigcap_{b=1}^B \left\{ \mathrm{sign}\left( P_{ab}(\underline{x}_m, \epsilon) \right) = \sigma_{ab} \right\},
\]
where \( \sigma_{ab} \in \{-1, 0, 1\} \) and each \( P_{ab} \) is a polynomial with integer coefficients. Thus, verifying whether a particular sample satisfies the constraints reduces to evaluating these polynomials at \( (\underline{x}_m, \epsilon) \) and checking whether the corresponding sign conditions hold.

While computationally demanding, this method is provably complete and forms the theoretical foundation for algorithmic decision; see \cite{BasuPollackRoy}. By contrast, even reducing empirical risk minimization to an o-minimal class does not ensure decidability unless the class is effective—a limitation underscored by Richardson’s Theorem (see, e.g., \cite{Richardson}).

\subsubsection{Approximation-based alternative via covering.}
The structure of the class $\HH(n,k,d)$ allows for a second, more algorithmic approach based on covering arguments. The quantitative precompactness result in \cref{coro:netintro} ensures that $\HH(n,k,d)$ can be covered by balls centered at finitely many smooth varieties of degree at most $2d$, each of which lies at controlled Hausdorff distance from any target variety in the class. This leads to the following conceptual strategy:

\begin{enumerate}
    \item Construct a finite $\epsilon$-net $\{Z_1, \ldots, Z_N\} \subset \HH(n,k,d)$, with $Z_i = Z(p^{(i)})$ and $\deg(p^{(i)}_j) \leq 2d$;
    \item For each $Z_i$, compute the empirical risk $\RR(Z_i, \mu_m)$;
    \item Accept the variety hypothesis if there exists $Z_i$ such that $\RR(Z_i, \mu_m) \leq \epsilon$ and $Z_i \cap D^n \neq \emptyset$.
\end{enumerate}

\begin{remark}
A key computational subroutine in both strategies is the evaluation of $\mathrm{dist}(x, Z)$ for given $x \in D^n$ and real algebraic set $Z$. This appears in the empirical risk and in geometric regularity conditions (e.g., distance from the discriminant). The algebraic complexity of this step is governed by the Euclidean Distance Degree (EDD), which bounds the number of critical points of the squared distance function from a generic point to $Z$; see \cite{EDD}. The EDD provides a measure of the intrinsic metric complexity of $Z$ and plays a central role in algorithmic aspects of metric algebraic geometry.

\end{remark}

In summary, \cref{thm:Nvariationintro} implies that the variety hypothesis problem can be reduced to a semialgebraic decision task involving a finite number of samples. One may approach this problem either via general--purpose tools from real algebraic geometry or through a concrete covering strategy using finite approximations. Both perspectives illustrate how statistical learning and algebraic geometry can be combined to rigorously address geometric inference problems over algebraic models. The description of a more direct and efficient algorithm, and the study of
its complexity, will be the subject of a future publication.

\section{Basic notions from manifold learning theory}

This section reviews some basic notions from statistical and manifold learning theory. For the former, standard references include \cite{VladimirVapnik} and \cite{TongZhang}; for the latter, see \cite{Narayanan} and \cite{Fefferman}. 
\subsection{Learning algorithms and learnability}\label{sec:learnability}Let $\HH$ be a family of compact subsets of $D^n$, called the hypotheses class. For a Borel probability measure $\mu$ on $D^n$, we consider the \emph{risk} functional $\RR(\cdot, \mu):\HH\to \R$, defined by
\[ \label{eq:risk}\RR(Z, \mu):=\int_{D^n}\mathrm{dist}(x, Z)^2\mu(\mathrm{d}x).\]
This is a notion of how far the measure $\mu$ is from being supported on $Z$, for instance if the support of $\mu$ is a subset of $Z$, then $\RR(Z, 
\mu)=0$. For a given $\mu$, we refer to the term 
\[\RR^*:= \inf _{Z \in \mathcal{H}} \RR(Z,\mu)\]
as the \textit{optimal risk on the class} $\mathcal{H}$. 

For every $m\in \N$ we denote by $\mathcal{S}_m:=(D^n)^m$ the $m$--samples space and we define the total sample space 
\[\mathcal{S}:=\bigsqcup_{m=1}^\infty \mathcal{S}_m.\]
 A \textit{learning algorithm} is a computable map $\mathcal{A}:\mathcal{S} \to \mathcal{H}$. In other words, it is an algorithm that takes as input a finite sequence of training samples and outputs a hypothesis in $\mathcal{H}$.  On each sample space  $\mathcal{S}_m$ we have the natural product probability measure\footnote{Here and below we will use the same symbol ``$\mathbb{P}$'' to denote the probability of events with respect to the product measure on each $\mathcal{S}_m$, the case $m=1$ being just $\mu$.} associated with the measure $\mu$ and the restriction of the algorithm $\mathcal{A}$ to each $\mathcal{S}_m$ defines a random variable $Z_{m}:=\mathcal{A}\left(\underline{x}_m\right)$ with values in $\mathcal{H}$, where  $\underline{x}_m=(x^1, \ldots, x^m) \in \mathcal{S}_m$ consists of $m$ independent samples from $\mu$. A given algorithm $\mathcal{A}$ is \textit{consistent} if for all $\epsilon>0$ and $0<\delta<1$, there exists $m_{0}\in \N$ such that for all $m \geq m_0$ \begin{equation}
\mathbb{P}\bigg\{\RR(Z_m, \mu) -\inf _{Z \in \mathcal{H}} \RR(Z,\mu) > \varepsilon \bigg\}<\delta \label{cons}. \end{equation}
Thus, an algorithm is consistent if it outputs with high probability a hypothesis with almost optimal risk. The \textit{sample complexity} $m^*(\epsilon, \delta,\mathcal{A})$ of the algorithm $\mathcal{A}$ in learning the class $\mathcal{H}$ is then the minimum $m_0$ for which the condition \eqref{cons} holds. If $\mathcal{A}$ is not consistent, we set $m^*(\epsilon, \delta,\mathcal{A})=\infty$.
If there exists an algorithm for which $m^*( \epsilon, \delta,\mathcal{A})$ is finite for all choices of $\mu, \epsilon,\delta$ then we say that the hypothesis space $\mathcal{H}$ is \textit{learnable.}

\subsection{Empirical risk minimization class of algorithms}

Typical learning algorithms include the family of empirical risk minimization algorithms (ERM). Algorithms in such class are based on the so called \textit{empirical risk minimization principle}. The core idea is to consider, for a given list $\underline{x}_m=(x^1, \ldots, x^m)\in (D^n)^m$  the \emph{empirical measure} 
\[\mu_m:=\frac{1}{m}\sum_{i=1}^k\delta_{x^i}.\]
and look for solutions of the optimization problem given by the risk functional \eqref{eq:risk} by actually minimizing the \textit{empirical risk}, defined to be the value ot the risk on the empirical measure
\[\RR(Z, \mu_m)=\frac{1}{m}\sum_{i=1}^m\mathrm{dist}(x^i, Z)^2.\]
Note that, if  $\underline{x}_m$ consists of a list of independent and identically distributed samples from  a probability measure $\mu$, then the value $\RR(Z, \mu_m)$ is a random variable, namely the  empirical risk is a random functional on $\HH$. 

Because in general the minimum of the risk might not exist,  ERM algorithms, for every $\epsilon>0$ and $\underline{x}_m\in \mathcal{S}_m$, output a hypothesis $\widehat{Z}_{\underline{x}_m}(\epsilon)$ with the property that
\[\label{ERM}\RR(\widehat{Z}_{\underline{x}_m}(\epsilon), \mu_m)-\inf_{Z\in \HH}\RR(Z, \mu_m)\leq \epsilon.\]
We call such a hypothesis $\widehat{Z}_{\underline{x}_m}(\epsilon)$ an $\epsilon$--minimizer for the ERM problem. The difference in the algorithms following the ERM principle is then in the different strategy they use to solve the optimization problem \eqref{ERM}. 

\subsection{More on the sample complexity}\label{sec:sample}

As noted earlier, \cref{thm:Nvariationintro} provides a quantitative upper bound on the sample complexity for empirical risk minimization (ERM) over the class $\HH(n,k,d)$. This should be contrasted with bounds derived from classical covering number arguments, such as those in \cite[Theorem 4]{Narayanan} and \cite[Theorem 1]{Fefferman}, which apply to general hypothesis classes $\HH$ admitting uniform $\epsilon$-coverings.

More precisely, \cite[Theorem 4]{Narayanan} proves the following general result:

\begin{theorem}[Narayanan]\label{thm:N1}
Let $\HH$ be any family of subsets of the unit disk $D^n$ such that there exists a function $N_{\mathcal{H}}:\R^+\to \mathbb{N}$ with the property that for every $Z\in \HH$, the number of $\epsilon$--balls needed to cover $Z$ is bounded by $N_\mathcal{H}(\epsilon)$. Then for every $\epsilon>0$, $0<\delta<1$, if
\[
m\geq C\left(\frac{N_{\mathcal{H}}(16\epsilon)}{\epsilon^2}\min\left\{N_{\mathcal{H}}(16\epsilon), \frac{1}{\epsilon^2}\log^4\left(\frac{N_\mathcal{H}(16\epsilon)}{\epsilon}\right)\right\}+\frac{1}{\epsilon^2}\log\frac{1}{\delta}\right)=:m(\epsilon, \delta, N_{\mathcal{H}}),
\]
then
\[
\mathbb{P}\left\{\sup_{Z\in \HH}\left|\RR(Z, \mu)-\widehat{\RR}_{m}(\underline{x}_m, Z)\right|\leq \epsilon\right\}\geq 1-\delta.
\]
Here $C>0$ is a universal constant.
\end{theorem}

This result is applied in \cite{Narayanan} to the class of manifolds with bounded geometry:
\[ 
\HH_{\mathrm{sm}}(k,V, \rho):=\left\{Z\subset D^n\,\middle|\, \dim(Z)=k,\ \mathrm{vol}(Z)\leq V,\ \mathrm{reach}(Z)\geq \rho\right\}.
\] 
For this class, the covering number can be bounded as (see \cite[Section 3.2]{Narayanan}):
\[ 
N_{\HH_{\mathrm{sm}}}(\epsilon)=V\left(C\frac{k}{\min\{\epsilon, \rho\}}\right)^k.
\]

By contrast, the class $\HH(n,k,d)$ of real algebraic varieties does not, in general, admit uniform bounds on geometric quantities such as  reach or curvature. Nonetheless, recent results from metric algebraic geometry still allow us to obtain effective covering estimates. In particular, \cite[Theorem 2.6]{KileelZhang} establishes that the number of $\epsilon$--balls needed to cover any $Z \in \HH(n,k,d)$ satisfies
\[
N(\epsilon, Z)\leq \left(\frac{4kn^{\frac{3}{2}}}{\epsilon}\right)^k(4d)^{n-k},
\]
which gives an upper bound for $N_{\HH(n,k,d)}(\epsilon)$. Inserting this into \cref{thm:N1} yields the sample complexity bound
\[
m\leq C\left(\frac{(kn^{\frac{3}{2}})^kd^{n-k}}{4^{2k-n}\epsilon^{k+2}}\min\left\{\frac{(kn^{\frac{3}{2}})^kd^{n-k}}{4^{2k-n}\epsilon^{k}}, \frac{1}{\epsilon^2}\log^4\left(\frac{(kn^{\frac{3}{2}})^kd^{n-k}}{4^{2k-n}\epsilon^{k+1}}\right)\right\}+\frac{1}{\epsilon^2}\log\frac{1}{\delta}\right).
\]

As discussed in \cref{sec:introexisting}, this bound makes explicit the dependence on $(n,k,d)$, but the dependence on $\epsilon$ is $O(\epsilon^{-k-4})$, which is significantly worse than the $O(\epsilon^{-2}\log(\epsilon^{-1}))$ scaling obtained in \eqref{eq:boundnkdintro}. On the other hand, our bound does not make the dependence on $(n,k,d)$ fully explicit, although this could be refined using further geometric estimates.

\medskip

It is also instructive to compare these two approaches from a geometric standpoint. One may attempt to connect the class $\HH(n,k,d)$ with the bounded--geometry manifold class $\HH_{\mathrm{sm}}(k,V,\rho)$ by using structural properties of smooth complete intersections. Propositions~\ref{propo:boundvolume} and~\ref{propo:reach} show that a smooth complete intersection $Z=Z(p_1,\dots,p_{n-k})$ with $p \in P(d)$ and $\mathrm{dist}(p,\Sigma)\geq \tau \|p\|$ satisfies explicit bounds on volume and reach, namely:
\[
\mathrm{vol}(Z \cap D^n) \leq \mathrm{vol}(D^k) \cdot d^{n-k}, \quad \text{and} \quad \mathrm{reach}(Z \cap D^n) \geq a_3(n,k,d)\tau.
\]
This suggests that, for smooth varieties bounded away from the discriminant, one may construct a formal dictionary identifying them with elements of $\HH_{\mathrm{sm}}(k,V,\rho)$ for suitable values of $V$ and $\rho$.

However, this dictionary--based comparison has two significant drawbacks: first, it only applies to smooth complete intersections that are quantitatively far from the discriminant, excluding possibly the most interesting part of $\HH(n,k,d)$; second, even when applicable, the resulting sample complexity bounds—derived by inserting the associated $N_{\HH_{\mathrm{sm}}}(\epsilon)$ into \cref{thm:N1}—are still worse than those obtained via our direct probabilistic analysis in \cref{thm:Nvariationintro}. In particular, they also scale as $O(\epsilon^{-k-4})$ as $\epsilon \to 0$.
\medskip

This comparison highlights both the flexibility and sharpness of our approach: by exploiting algebraic structure directly—rather than relying on uniform geometric regularity—we obtain stronger statistical guarantees and broader applicability within the realm of real algebraic varieties.

\section{Geometry of the space of polynomial equations}
\subsection{Reminders on real and semialgebraic geometry}\label{sec:semialgebraic}

We collect here some basic definitions and conventions from real and semialgebraic geometry, which will be used throughout the paper. Standard references for this material include \cite{BCR, BasuPollackRoy}.

A \emph{real algebraic set} is a subset \( Z \subset \R^n \) defined as the common zero locus of finitely many real polynomials:
\[
Z = \{ x \in \R^n \mid p_1(x) = \cdots = p_s(x) = 0 \}, \quad p_i \in \R[x_1, \dots, x_n].
\]
These sets are closed in the Euclidean topology and stable under finite unions and intersections. A real algebraic set that is also irreducible (i.e., cannot be written as the union of two proper algebraic subsets) is called an \emph{algebraic variety}. In this paper, we use the term ``variety'' informally to refer to general real algebraic sets, without requiring irreducibility.

We say that a real algebraic set \( Z \) has \emph{degree at most} \( d \) if it can be described as the zero set of polynomials \( p_1, \ldots, p_s \) all of degree at most \( d \). This notion provides an upper bound on the algebraic complexity of the set, although it does not uniquely determine the degree in a geometric or algebro--geometric sense. In particular, different defining systems may yield the same set with different degree bounds.

More generally, a \emph{semialgebraic set} is any subset \( S \subset \R^n \) that can be defined by a finite sequence of polynomial equalities and inequalities. Formally, \( S \) is semialgebraic if it is a finite union of sets of the form
\[
\{ x \in \R^n \mid f_1(x) = 0, \dots, f_k(x) = 0,\ g_1(x) > 0, \dots, g_\ell(x) > 0 \},
\]
for some polynomials \( f_i, g_j \in \R[x_1, \dots, x_n] \). Semialgebraic sets form a Boolean algebra: they are closed under finite unions, intersections, and complements, and are also stable under projections. Their structure admits a well--defined dimension theory and stratification into smooth manifolds.

For instance, if \( Z \subset \R^n \) is an algebraic set, then \( Z \cap D^n \) is typically no longer algebraic but remains semialgebraic. Its dimension can be characterized informally as the smallest integer \( k \) such that a generic affine subspace of codimension \( k+1 \) does not intersect \( Z \cap D^n \). More formally, semialgebraic sets admit finite \emph{semialgebraic stratifications}—that is, decompositions into finitely many disjoint, connected, smooth semialgebraic manifolds (called strata), each of a well-defined dimension. The dimension of a semialgebraic set is then defined as the maximum dimension among its strata. This notion coincides with the algebraic dimension when the set is purely algebraic and is stable under semialgebraic operations such as projection and Boolean combinations. 

A \emph{real algebraic set} \( Z \subset \R^n \) is called a \emph{complete intersection} of dimension \( k \) if it can be written as the common zero set of exactly \( c = n - k \) polynomials \( p_1, \dots, p_c \in \R[x_1, \dots, x_n] \), and if the gradients \( \nabla p_1(x), \dots, \nabla p_c(x) \) are linearly independent at every point \( x \in Z \). In this case, \( Z \) is said to be a \emph{smooth complete intersection}, and is a smooth manifold of dimension \( k \) near each point.

It is often convenient—particularly for algebraic and algorithmic purposes—to encode such a set using a single nonnegative polynomial:
\[
p = p_1^2 + \cdots + p_c^2,
\]
so that \( Z = \{x \in \R^n \mid p(x) = 0\} \). This representation is always possible over \( \R \), and it defines the same underlying set. However, it may increase the degree (up to twice the maximum degree of the \( p_i \)) and does not preserve structural properties such as smoothness or transversality. This trade-off is often acceptable when the goal is to describe or search over algebraic sets semialgebraically, as in the algorithms we discuss later in the paper.

\subsection{Geometry in the space of polynomials}
Denote by $\Pol$ the set of $c$--tuples $p=(p_1, \ldots, p_c)$ of real polynomials in $n$ variables and degree $d$. This is a real vector space of dimension $c\binom{n+d}{d},$ i.e. the space of the coefficients of the polynomials. Each element $p\in \Pol$ defines a real algebraic set in $\R^n$ denoted, as above, by 
\[Z(p):=\bigg\{x\in \R^n\, \bigg|\, p_1(x)=\cdots=p_c(x)=0\bigg\}.\]
We will only be interested in behavior of $Z(p)$ inside the unit disk in $\R^n$. To this end, first we restrict to a closed semialgebraic subset of $\Pol$, namely the set:
\[\label{eq:defpd} P(d):=\bigg\{p\in \Pol\,\bigg|\, Z(p)\cap D^n\neq \emptyset\bigg\}.\]
(This set is semialgebraic, since so is the condition ``there exists $x\in D^n$ such that $p(x)=0$''.)
We omit the dependence on $n$ and $k=n-c$ in the notation. We also consider the set of $c$--tuples giving nonregular equations, called the \emph{discriminant}.
\begin{definition}[The discriminant]\label{def:discriminant}We denote by $\Sigma(d)\subset P(d)$ the set consisting of $c$--tuples $(p_1, \ldots, p_d)$ such that one of the following conditions is verified:
\begin{enumerate}
\item there exists $x\in D^n$ such that $p(x)=0$ and $\mathrm{rk}(D_xp)\leq c-1$.
\item there exists $x\in \partial D^n$ such that $p(x)=0$ and $\mathrm{rk}(D_xp|_{x^\perp})\leq c-1.$
\end{enumerate}
The two conditions above are semialgebraic and express the possibility that the equation $p=0$ \emph{is not} regular on the disk $D^n$, seen as a manifold with boundary.
We denote by 
\[ T(d):=P(d)\setminus \Sigma(d)\]
the set of regular equations. 
\end{definition}
Every element $p\in T(d)$ has the property that the equation $p=0$ is regular both on $D^n$ and  on $\partial D^n$. In particular, the zero set $Z(p)\cap D^n$ is a smooth (nonempty because $p\in P(d)$) manifold of dimension $k=n-c$ with (possibly empty) smooth boundary $Z(p)\cap \partial D^n$.

\subsection{Distance to the discriminant}
We can quantify the amount of regularity of a zero set as follows. We first endow the space of polynomials with a Euclidean structure, the Bombieri--Weyl for this paper (\cref{def:BW}),  and we consider the induced distance. The closer the defining polynomials are to  the discriminant, the less regular their zero set is.

\begin{definition}\label{def:BW}Write the components of a polynomial $c$--tuple $p=(p_1, \ldots, p_c)$ as
\[ p_j(x)=\sum_{|\alpha|\leq d}p_{j, \alpha}\left(\frac{d!}{(d-|\alpha|)!\alpha_1!\cdots \alpha_n!}\right)^{1/2}x_1^{\alpha_1}\cdots x_n^{\alpha_n}\]
(i.e. the $p_{j, \alpha}\in \R$ are the coefficient of $p_j$ in a rescaled monomial basis).
We define the Bombieri--Weyl scalar product on $P(d)$ as:
\[\langle p, q\rangle_{\mathrm{BW}}:=\sum_{j=1}^c\sum_{|\alpha|\leq d}p_{j, \alpha}q_{j, \alpha}.\]
We denote the corresponding norm simply by $\|\cdot\|_{\mathrm{BW}}$. For a closed set $\Sigma\subseteq P(d)$ we denote by $\mathrm{dist}(\Sigma, \cdot)$ the distance from this set
\[\mathrm{dist}(\Sigma, p):=\inf_{q\in \Sigma}\|q-p\|_{\mathrm{BW}}.\]
\end{definition}
The goal of this section is to prove the following result, which is inspired by \cite[Theorem 5.1]{Raffalli}, where the author computes the distance in the Bombieri--Weyl norm from a homogeneous polynomial to the set of polynomials with a singular zero on the sphere.

\begin{proposition}\label{propo:distdiscgen}For every $n, k,d$ there exists $a(n,k,d)>0$ such that:
\begin{align}\label{discineqstatement}
\mathrm{dist}(p, \Sigma)\leq a(n,k,d)\cdot \min\bigg\{&\min_{\|z\|\leq 1}\left(\|p(z)\|^2+\sigma_{\min}(Jp(z))^2\right)^{1/2}, \\
&\min_{\|z\| =1}\left(\|p(z)\|^2+\sigma_{\min}(Jp(z)|_{z^\perp})^2\right)^{1/2}\bigg\},
\end{align}
where $\sigma_{\min}(Jp(z))$ denotes the smallest singular value of the matrix $Jp(z)$.
\end{proposition}
In order to prove the result it will be convenient to introduce some additional notation.
First we observe that $\Sigma$ is actually the union of two smaller discriminants:
\[\label{eq:sigmadeco}\Sigma=\Sigma_{D}\cup \Sigma_{\partial D},\]
where 
\[\label{eq:SigmaD}\Sigma_{D}:=\left\{p\in P(d)\,\bigg|\, \exists x\in D^n,\, p(x)=0,\, \mathrm{rk}(D_xp)\leq c-1\right\}\]
and 
\[\label{eq:SigmapartialD}\Sigma_{\partial D}:=\left\{p\in P(d)\,\bigg|\, \exists x\in \partial D^n,\, p(x)=0,\, \mathrm{rk}(D_xp|_{x^\perp})\leq c-1\right\}.\]
Moreover, these two discriminants can be further described as
\[\Sigma_{D}=\bigcup_{x\in D^n}\Sigma_{D, x}\quad \textrm{and}\quad \Sigma_{\partial D}=\bigcup_{x\in \partial D}\Sigma_{\partial D, x},\]
where 
\[\Sigma_{D, x}:=\left\{p\in P(d)\,\bigg|\, p(x)=0,\, \mathrm{rk}(D_xp)\leq c-1\right\}\]
and 
\[\Sigma_{\partial D, x}:=\left\{p\in P(d)\,\bigg|\, p(x)=0,\, \mathrm{rk}(D_xp|_{x^\perp})\leq c-1\right\}.\]

Before giving the proof of \cref{propo:distdiscgen}, we first prove \cref{propo:distdisc} and \cref{propo:distdisc2},  which give bounds on the distance from the two sets $\Sigma_{D}$ and $\Sigma_{\partial D}$. We do this separately in the next two sections, starting from the first.

\subsubsection{Distance to $\Sigma_D$}
In order to study the distance to $\Sigma_D$, denote by $\widehat{p}:S^n\to \R^c$ the restriction of the homogenization of $p\in P(d)$ to the unit sphere $S^n\subset \R^{n+1}$, i.e.
\[\widehat{p}:=\left(\sum_{|\alpha|\leq d}p_{ \alpha}\left(\frac{d!}{(d-|\alpha|)!\alpha_1!\cdots \alpha_n!}\right)^{1/2}x_0^{d-|\alpha|}x_1^{\alpha_1}\cdots x_n^{\alpha_n}\right)\Bigg|_{\left\{x_0^2+\cdots +x_n^2=1\right\}}.\]
In this way elements of $P(d)$ become functions on $S^n$ and one can consider  the corresponding discriminant
\[\label{eq:discriminant2}\widehat{\Sigma}=\widehat{\Sigma}(d):=\left\{p\in P(d)\,\bigg|\, \exists u\in S^n: \,  \widehat{p}(u)=0,\,\mathrm{rk}(D_u\widehat{p})\leq c-1\right\}.\]
Elements from $\widehat{\Sigma}$ are polynomial maps $\widehat{p}:S^n\to \R^c$ such that the equation $\widehat{p}=0$ is not regular on the sphere.
Let now $u\in S^n$; we denote by $\widehat{\Sigma}_u\subset \widehat{\Sigma}$ the set
\[\widehat{\Sigma}_u:=\left\{p\in P(d)\,\bigg|\, \widehat{p}(u)=0,\,\mathrm{rk}(D_u\widehat{p})\leq c-1\right\}.\]
In this way we get a description of $\widehat{\Sigma}$ as the union 
\[\label{eq:union}\widehat{\Sigma}=\bigcup_{u\in S^n}\widehat{\Sigma}_u.\]
The two discriminants $\Sigma_D$ defined by \eqref{eq:SigmaD} and $\widehat{\Sigma}$ defined by \eqref{eq:discriminant2} are related as follows.
We denote by $\pi:\R^n\to S^n$ the (inverse of the) stereographic projection,
\[\pi(x):=\frac{1}{\sqrt{1+\|x\|^2}}(1, x).\]
\begin{lemma}\label{lemma:sigmatilde}For every $x\in \R^n$, we have $\Sigma_{x}=\widehat{\Sigma}_{\pi(x)}.$
In particular,
\[\label{eq:sigmatilde}\Sigma_D=\bigcup_{\|x\|\leq 1}\widehat{\Sigma}_{\pi(x)}.\]
\end{lemma}
\begin{proof}First observe as a consequence of the definition of $\widehat{p}$ the following identity:
\[\widehat{p}(\pi(x))=\frac{p(x)}{(1+\|x\|^2)^{2d}}.\]
From this identity we get an identity involving the differentials:
\[D_{\pi(x)}\widehat{p}D_x\pi=D_x\left((1+\|x\|^2)^{-2d}\right)p(x)+\left(1+\|x\|^2\right)^{-2d}D_xp.\]
In particular, since $D_x\pi$ is an isomorphism, the two conditions 
\[p(x)=0 \quad\textrm{and}\quad \mathrm{rk}(D_xp)\leq c-1\] are both satisfied if and only if so are the two conditions
\[\widehat{p}(\pi(x))=0\quad \textrm{and} \quad \mathrm{rk}(D_{\pi(x)}\widehat{p})=\mathrm{rk}\left(\left(1+\|x\|^2\right)^{-2d}D_xp\right)\leq c-1.\]
This means $\Sigma_{x}=\widehat{\Sigma}_{\pi(x)}$. The second part of the statement is an immediate consequence of the first part and the definition of $\Sigma$.
\end{proof}

The following result describes the distance, with respect to the Bombieri--Weyl scalar product, from a point $p\in P(d)$ to $\widehat{\Sigma}(d)$. In the case $c=1$ this is a theorem of Raffalli, \cite[Theorem 5.1]{Raffalli}. As we will see, the proof from \cite{Raffalli} easily generalizes to the case $c>1$. In order to state the result, for a linear map $L:V\to W$ between Euclidean vector spaces of dimension $\dim(V)=n$, $\dim(W)=c$, we denote by $\sigma_{\min}(L)\in [0, \infty)$ the smallest singular value of a matrix $M\in \R^{c\times n}$ representing $L$ with respect to orthonormal bases for $V$ and $W$. In the case of our interest
\[L=D_u\widehat{p}:T_uS^n\to \R^c,\]
where both spaces are endowed with their natural Euclidean structure.

\begin{proposition}\label{propo:invariance}Let $p\in P(d)$ and $u\in S^n$. Then
\[\label{eq:distu}\mathrm{dist}(p, \widehat{\Sigma}_u)=\left(\|\widehat{p}(u)\|^2+\frac{\sigma_{\min}\left(D_u\widehat{p}\right)^2}{d}\right)^{\frac{1}{2}}.\]
In particular,
\[\label{eq:raffalli}\mathrm{dist}(p, \widehat{\Sigma})=\min_{u\in S^n}\left(\|\widehat{p}(u)\|^2+\frac{\sigma_{\min}\left(D_u\widehat{p}\right)^2}{d}\right)^{\frac{1}{2}}.\]
\end{proposition}
\begin{proof}
The orthogonal group $O(n+1)$ acts linearly on the space of (homogeneous) polynomials by change of variables and the Bombieri--Weyl scalar product is invariant under this action. Using this and \eqref{eq:union}, the problem reduces to compute $\mathrm{dist}(\widehat{p}, \Sigma_{e_0})$, where $e_0=(1, 0, \ldots, 0)$, and this can be done explicitly as in \cite{Raffalli}. The second part of the statement is a direct consequence of the first part.
\end{proof}
As a consequence of the previous result we get the following.
\begin{proposition}\label{propo:distdisc}For every $n,k,d\in \N$ there exists $a_1(n,k,d)>0$ such that:
\begin{equation}
\mathrm{dist}(p, \Sigma_D)\leq a_1(n,k,d)\min_{\|z\|\leq 1}\left(\|p(z)\|^2+\sigma_{\min}(D_zp)^2\right)^{1/2}.
\label{discineq}
\end{equation}
\end{proposition}
\begin{proof}The equations \eqref{eq:sigmatilde} and \eqref{eq:distu} imply that
\[\mathrm{dist}(p, \Sigma_D)=\min_{\|x\|\leq 1}\left(\|\widehat{p}(\pi(x))\|^2+\frac{\sigma_{\min}\left(D_{\pi(x)}\widehat{p}\right)^2}{d}\right)^{\frac{1}{2}}.\]
Recall now the following two identities from the proof of \cref{lemma:sigmatilde}:
\[\widehat{p}(\pi(x))=\frac{p(x)}{(1+\|x\|^2)^{2d}}\]
and
\[D_{\pi(x)}\widehat{p}=D_{x}\pi^{-1}\left(D_x\left((1+\|x\|^2)^{-2d}\right)p(x)+\left(1+\|x\|^2\right)^{-2d}D_xp\right). \]
Since $\|x\|\leq 1$, the singular values of $D_{x}\pi^{-1}$ can be bounded by a function of the norm of $x$ only. Similarly, $(1+\|x\|^2)^{-2d}$ and its derivative can be bounded by $O(d\|x\|)$. This implies the statement. 
\end{proof}
\subsubsection{Distance to $\Sigma_{\partial D}$}The distance to $\Sigma_{\partial D}$ can be estimated via the folllowing.
\begin{proposition}\label{propo:distdisc2}For every $n,k,d\in \N$ there exists $a_2(n,k,d)>0$ such that
\[\mathrm{dist}(p, \Sigma_{\partial D})\leq a_2(n,k,d)\min_{\|z\| =1}\left(\|p(z)\|^2+\sigma_{\min}(D_zp|_{z^\perp})^2\right)^{1/2}.\]
\end{proposition}
\begin{proof}We will use the description
\[\label{eq:union}\Sigma_{\partial D}=\bigcup_{z\in \partial D}\Sigma_{\partial D, z}\]
and first compute the distance of $p$ to each $\Sigma_{\partial D, z}.$ First observe that the orthogonal group $O(n)$, the stabilizer of the point at infinity for the homogenization of polynomials, acts by isometries on the space $\Pol$ with the Bombieri--Weyl norm (because it is a subgroup of $O(n+1)$ that acts by isometries, as in \cref{propo:invariance}). Note that the action on restrictions of polynomials to $\partial D\simeq S^{n-1}$ is then the usual action of $O(n)$ on the space of functions on the sphere $S^{n-1}$ by change of variables. In particular, denoting by $R\in O(n)$ an orthogonal transformation such that $Rz=e_1$, we have
\[\mathrm{dist}(p, \Sigma_{\partial D, z})=\mathrm{dist}(p\circ R^{-1}, \Sigma_{\partial D, e_1}).\]
Now, in the monomial bases, we see that $\Sigma_{\partial D, e_1}$ coincides with the set of polynomials $q:\R^n\to \R^c$ such that \[q(x)=L(x-e_1)+\textrm{h.o.t.},\] where $L=Jq(e_1)\in \R^{c\times n}$ is a matrix with the property that $\mathrm{rk}(L|_{e_1^\perp})\leq c-1.$

Denote now by $\tilde{p}:=p\circ R^{-1}$ and consider the Taylor expansion of $\tilde{p}$ at $e_1$:
\[\tilde{p}(x)=\tilde{p}(e_1)+D_{e_1}\tilde{p}(x-e_1)+\textrm{h.o.t}.\] 
Then, using the Taylor expansions of $q$ and $\tilde{p}$,
\begin{align}\mathrm{dist}(p, \Sigma_{\partial D, z})=&\mathrm{dist}(\tilde{p}, \Sigma_{\partial D, e_1})\\
=&\min_{q\in \Sigma_{\partial, e_1}}\|\tilde{p}-q\|\\
\leq& \min_{q\in \Sigma_{\partial, e_1}}\|\tilde{p}(e_1)+(D_{e_1}\tilde{p}-L)(x-e_1)\|\\
\leq& \|\tilde{p}(e_1)\|+\min_{q\in \Sigma_{\partial, e_1}}\|(D_{e_1}\tilde{p}-L)(x-e_1)\|\\
=&\|\tilde{p}(e_1)\|+\min_{\mathrm{rk}(L|_{e_1^\perp})\leq c-1}\|(D_{e_1}\tilde{p}-L)(x-e_1)\|\\
=&\|\tilde{p}(e_1)\|+\min_{\mathrm{rk}(L|_{e_1^\perp})\leq c-1}\bigg(\|(D_{e_1}\tilde{p}-L)\mathrm{proj}_{e_1^\perp}(x-e_1)\|+\\
&+\|(D_{e_1}\tilde{p}-L)\mathrm{proj}_{e_1}(x-e_1)\|\bigg)\\
=&\|\tilde{p}(e_1)\|+\min_{\mathrm{rk}(Q)\leq c-1}\|(D_{e_1}\tilde{p}|_{e_1^\perp}-Q)\mathrm{proj}_{e_1^\perp}(x-e_1)\|\\
=&\|\tilde{p}(e_1)\|+C(n,k,d)\min_{\mathrm{rk}(Q)\leq c-1}\|(D_{e_1}\tilde{p}|_{e_1^\perp}-Q)\mathrm{proj}_{e_1^\perp}(x-e_1)\|_{C^1(D)}\\
\leq&\|\tilde{p}(e_1)\|+C(n,k,d)\min_{\mathrm{rk}(Q)\leq c-1}\|(D_{e_1}\tilde{p}|_{e_1^\perp}-Q)\|_2\|\mathrm{proj}_{e_1^\perp}(x-e_1)\|_2\\
\leq&\|\tilde{p}(e_1)\|_2+C(n,k,d)\min_{\mathrm{rk}(Q)\leq c-1}\|(D_{e_1}\tilde{p}|_{e_1^\perp}-Q)\|_2\\
\leq&\|\tilde{p}(e_1)\|_2+C(n,k,d)\sigma_{\min}(D_{e_1}\tilde{p}|_{e_1^\perp})\\
\leq& a_2(n,k,d)\left(\|\tilde{p}(e_1)\|^2+\sigma_{\min}(D_{e_1}\tilde{p}|_{e_1^\perp})^2\right)^{\frac{1}{2}}\\
\label{eq:last}=&a_2(n,k,d)\left(\|p(z)\|^2+\sigma_{\min}(D_{z}p|_{z^\perp})^2\right)^{\frac{1}{2}}.
\end{align}
Using now \eqref{eq:union} and \eqref{eq:last}, we get
\[\mathrm{dist}(p, \Sigma_{\partial D})=\min_{\|z\|=1}\mathrm{dist}(p, \Sigma_{\partial D, z})\leq a_2(n,k,d)\min_{\|z\| =1}\left(\|p(z)\|^2+\sigma_{\min}(D_zp|_{z^\perp})^2\right)^{1/2},\]
as in the statement.
\end{proof}

We are now ready to give the proof of \cref{propo:distdiscgen}.
\begin{proof}[Proof of \cref{propo:distdiscgen}]The decomposition \eqref{eq:sigmadeco} implies that
\[\label{eq:deco1}\mathrm{dist}(p, \Sigma)=\min\left\{\min_{x\in D^n}\mathrm{dist}(p, \Sigma_{D, x}), \min_{x\in \partial D^n}\mathrm{dist}(p, \Sigma_{\partial D, x})\right\},\]The result follows immediately from the decomposition \eqref{eq:deco1} and \cref{propo:distdisc} and \cref{propo:distdisc2} below.
\end{proof}
\subsection{Bounds on the geometry of real complete intersections}\label{sec:boundcomplete} In this section we prove two results giving explicit bounds on the geometry of complete intersections. These results are not needed for the rest of the paper and serve just to establish a practical dictionary between algebraic and geometric regularity conditions, as discussed in \cref{sec:sample}.
\begin{proposition}\label{propo:boundvolume}Let $Z=Z(p_1, \ldots, p_{n-k})$ be a smooth complete intersection in $\R^n$ with $p=(p_1, \ldots, p_{n-k})\in P(d)$. Then
\[\mathrm{vol}(Z\cap D^n)\leq \mathrm{vol}(D^k)d^{n-k}.\]
\end{proposition} 
\begin{proof}Denote by $A(n-k, n)$ the Grassmannian of affine $(n-k)$--planes in $\R^n$, endowed with the  translation invariant measure $\gamma$. Then, for any compact semialgebraic set $S$ of dimension $k$, by Integral Geometry (see \cite[Chapter 5]{ComteYomdin})
\[\mathrm{vol}(S)=c(n,k)\int_{A(n-k, n)} \#\left(A\cap S\right)\gamma(\mathrm{d}A).\]
The normalization constant satisfies:
\[c(n,k)=\frac{\mathrm{vol}(D^k)}{\gamma\left(\left\{A\in A(n-k,n)\,\big|\,A\cap D^{k}\neq \emptyset\right\}\right)}.\]
Letting $S=Z\cap D^{n}$, we see that:
\begin{align}\mathrm{vol}(Z\cap D^n)&=c(n,k)\int_{A(n-k, n)} \#\left(A\cap Z\right)\gamma(\mathrm{d}A)\\
&\leq c(n,k)\int_{\{A\cap D^{k}\neq \emptyset\}}d^{n-k}\gamma(\mathrm{d}A)\\
&=\mathrm{vol}(D^k)d^{n-k},\end{align}
where we have used the fact that, by Bezout's Theorem, $\#(Z\cap A)\leq d^{n-k}$ for a generic affine plane $A$ of dimension $k$.
\end{proof}

\begin{proposition}\label{propo:reach}Let $Z=Z(p_1, \ldots, p_{n-k})$ be a smooth complete intersection in $\R^n$ with $p=(p_1, \ldots, p_{n-k})\in P(d)$ and such that $\mathrm{dist}(p, \Sigma)\geq \tau \|p\|.$ Then
\[\mathrm{reach}(Z\cap D^n)\geq a_3(n,k,d)\tau .\]
\end{proposition} 
\begin{proof}First observe that the condition $\mathrm{dist}(p, \Sigma)\geq \tau \|p\|,$ implies (via \cref{propo:distdiscgen})
\[\sigma_{\min}(Jp(z))\geq \frac{\tau \|p\|}{a(n,k,d)},\quad \forall z\in Z=Z(p).\]
Observe also that, using Taylor expansion, for every $z_1, z_2\in D^n$ there exists $\zeta_1, \ldots, \zeta_{n-k}$ on the segment joining $z_1 $ and $z_2$ such that
\[p(z_2)-p(z_1)=Jp(z_1)(z_1-z_2)+\tfrac{1}{2}\left((z_1-z_2)^\top(Hp_1(\zeta_1), \ldots, Hp_{n-k}(\zeta_{n-k}))(z_1-z_2)\right)^\top, \]
where, for $j=1, \ldots, n-k,$ we denote by $Hp_j(\zeta_j)$ the Hessian of $p_j$ at $\zeta_j$. In particular, if $z_1, z_2$ are now in $Z(p)$, from this identity we get
\[\label{eq:Jnorm}\|Jp(z_1)(z_1-z_2)\|\leq \|z_1-z_2\|^2\max_{\|z\|\leq 1}\|D_z^2p\|.\]Recall now that the reach of a smooth manifold $Z\subset \R^n$ can be computed as (see \cite[Theorem 4.18]{Federer}):
\[\label{eq:reach}\mathrm{reach}(Z)=\inf_{z_1, z_2\in Z,\, z_1\neq z_2}\frac{\|z_1-z_2\|^2}{2 \|\mathrm{proj}_{T_{z_1}Z^\perp}(z_1-z_2)\|},\]
where $\mathrm{proj}_{T_{z_1}Z^\perp}:\R^n\to T_{z_1}Z^\perp$ denotes the orthogonal projection. Since $\ker(Jp(z_1))=T_{z_1}Z$, then  $T_{z_1}Z^\perp=\mathrm{im}(Jp(z_1))$. Denote by $L:=Jp(z_1)$. Then, by the SVD decomposition, there exist orthogonal matrices $R_1\in O(n-k)$ and $R_2\in O(n)$ such that
\[L=R_1\left(\begin{array}{ccc|ccc}\sigma_1(L) &  &  & 0 & \cdots & 0 \\ & \ddots &  & \vdots &  & \vdots \\ &  & \sigma_{n-k}(L) & 0 & \cdots & 0\end{array}\right)R_2,\]
where $\sigma_{1}\leq \cdots\leq \sigma_{n-k}$  are the singular values of $L$ (here $\sigma_{\min}\equiv \sigma_1$).
In particular, 
\[\mathrm{proj}_{T_{z_1}Z^\perp}=R_1\left(\begin{array}{ccc|ccc}1 &  &  & 0 & \cdots & 0 \\ & \ddots &  & \vdots &  & \vdots \\ &  & 1& 0 & \cdots & 0\end{array}\right)R_2,\]
and, consequently, 
\[\label{eq:lower}\|\mathrm{proj}_{T_{z_1}Z^\perp}(z_1-z_2)\| \leq\sigma_{\min}(L)^{-1}\|L(z_1-z_2)\| \leq\sigma_{\min}(L)^{-1}\|z_1-z_2\|^2\max_{\|z\|\leq 1}\|D_z^2p\|,
\]
where we have used \eqref{eq:Jnorm} for the second inequality.

Getting back now to the reach estimation for $Z$, using \eqref{eq:lower} into \eqref{eq:reach}, we obtain
\begin{align}\mathrm{reach}(Z)&\geq \inf_{z_1, z_2\in Z,\, z_1\neq z_2}\frac{\sigma_{\min}(Jp(z_1))\|z_1-z_2\|^2}{2 \|z_1-z_2\|^2\max_{\|z\|\leq 1}\|D_z^2p\|}\\
&= \inf_{z_1, z_2\in Z,\, z_1\neq z_2}\frac{\sigma_{\min}(Jp(z_1))}{2 \max_{\|z\|\leq 1}\|D_z^2p\|}\\
&\geq \frac{\tau}{2 a(n,k,d) d}\frac{\|p\|}{\max_{\|z\|\leq 1}\|D_z^2p\|.}\\
&\geq \frac{\tau}{2 a(n,k,d) d^2}=a_3(n,k,d)\tau ,\end{align}
where in the last inequality we have used the fact that $\|p\|_{C^2(D^n)}\leq d\|p\|.$ 
\end{proof}

\subsection{A local Lipschitz constant}

Denote by $\mathcal{K}=\mathcal{K}(D^n)$ the set of nonempty compact subsets of $D^n$, endowed with the Hausdorff distance, defined for $C_1, C_2\in \mathcal{K}$ by:
\[\label{def:Hdist}\mathrm{dist}_H(C_1, C_2):=\max\left\{\max_{x\in C_1}\min_{y\in C_2}\|x-y\|,\max_{y\in C_2}\min_{x\in C_1}\|x-y\|\right\}.\] 
Recalling the definition of the set $P(d)$ from \eqref{eq:defpd}, we denote by
\[\kappa:P(d)\to \mathcal{K}\]
the map associating to each polynomial its zero set on the disk. The goal of this section is to prove the following Theorem, which proves that the map $\kappa$ is locally Lipschitz away from $\Sigma(d)$, with Lipschitz constant depending on the inverse of the distance from $\Sigma(d)$ itself.

\begin{theorem}\label{thm:thom}For every $n,k,d\in \N$ there exists $L>0$ such that, for every  $p\in T(d)$  and for $\|p-q\|$ small enough,
\[ \mathrm{dist}_H\left(\kappa(p), \kappa(q)\right)\leq \frac{L}{\mathrm{dist}(p, \Sigma(d))}\cdot \|p-q\|.\]
\end{theorem}

The key idea of the proof follows the classical Thom's Isotopy Lemma. Given a smooth family of functions $f_t: D^n \to \R^c $ parametrized by $t \in [0,1]$, the lemma  states that if 
\[f_t \pitchfork \{0\} \quad \textrm{and}\quad  {f_t}|_{\partial D}\pitchfork \{0\}  \quad \forall t \in [0,1],\] 
then there exists an isotopy $\phi_t: D^n \to D^n$ such that $\phi_t(Z(f_0))=Z(f_t)$ for all $t$. Denoting by $F(x,t)=f_t(x)$, the isotopy $\phi_t$ is defined as the flow of a non autonomous vector field $v_t(x) \in TD^n$ satisfying the constraints (see \cite[Section 2.1]{lecturenotesantonio} for more details):

\begin{enumerate}
\item\label{cond1} $\hat{v}(x,t):=v_t(x)+\partial_t$ is tangent to $\hat{Z}:=Z(F)\subset D^n\times I,$
i.e. $v_t(x)+\partial_t$ is in the kernel of $D_{(x,t)}F$;
\item\label{cond2} $\forall x \in \partial D^n$ satisfies  $v_t(x) \in x^{\perp}$
\end{enumerate}
\cref{thm:thom} is obtained by making this reasoning quantitative. We first prove the following technical, but crucial, result.

\begin{proposition}\label{propo:thom}For every $n,k,d\in \N$ there exists $a=a(n, k, d)>0$ (the same quantity as in \cref{propo:distdiscgen}) such that the following is true. Let $p\in P(d)$  and $\tau>0$ such that $\mathrm{dist}(p, \Sigma)\geq \tau \|p\|.$ Then, for every $f\in C^{1}(D^n, \R^c)$,
\[\label{eq:thombound}\|f-p\|_{C^1}\leq \frac{\tau \|p\|}{2} \implies\operatorname{dist}_{H}(Z(p)\cap D^n, Z(f)\cap D^n)\le  \frac{4a(n, k, d)}{\tau \|p\|}    \n{f-p}_{C^1}\]\end{proposition}

\begin{proof}

Define the function $F:  D^n \times I \to \R$ by $$F(x,t):=p(x)(1-t)+tf(x)$$ for $t \in [0,1]$ and set $f_t(x):=F(x,t)$. The two conditions \cref{cond1} and \cref{cond2} above read
\[
\label{tang}
D_{(x,t)}F(v_t(x)+\partial_t)=0 \quad\forall (x,t)\in D^n\times I \quad \text{and} \quad v_t(x)\in T_x\partial D^n \quad \forall x \in \partial D^n.
\]

We now prove that a vector field $v_t(x)$ satisfying both conditions exists. Notice first that, since $v_t(x)\in T_x\partial D^n$ for $x\in \partial D^n$, the flow of such vector field will preserve $D^n$.  Moreover if $\hat{v}$ is tangent to $\hat{Z}=Z(F)$,  its flow $\hat{\phi}$ will preserve $\hat{Z}$. 
We now build explicitly the vector field satisfying $\eqref{tang}$, solving the equations locally and then using a partition of unity to find the global vector field. 

Notice first the following. For $f:D^n\to \R^c$  and $z\in D^n$, denote by $\sigma_f (z)\in \R^{c+1}$ the vector:
\[\sigma_f(z):=(f(z), \sigma_1(Jf(z))).\]
Observe now that, because of the $1$--Lipschitz continuity of singular values, we can write
\[\sigma_{f_t}(z)=\sigma_p(z)+tw(z),\]
for a vector $w(z)\in \R^{c+1}$ with norm \[\|w(z)\|\leq \|f-p\|_{C^1}.\]
Arguing similarly for $f|_{\partial D^n}$, we see that,  if $\|f-p\|_{C^1}\leq \frac{\mathrm{dist}(p, \Sigma)}{2a(n,k,d)}$ then, 
by \cref{propo:distdiscgen}, we have
\begin{equation}
\begin{aligned}
\label{conds}
 &\frac{\tau \|p\|}{2a(n, k, d)} \leq \min_{\|z\|\leq 1}\|\sigma_{f_t}(z)\|= \min _{\|z\| \leq 1}\left(\|f_t(z)\|^2+\sigma_{\min}(J f_t(z))^2\right)^{1 / 2}\\
 & \frac{\tau \|p\|}{2a(n, k, d)} \leq    \min _{\|z\|=1}\left(\|f_t(z)\|^2+\sigma_{\min}\left(\left.J f_t(z)\right|_{z^{\perp}}\right)^2\right)^{1 / 2}.
\end{aligned}
\end{equation} We now proceed to define the local vector fields, in the following way:
\begin{enumerate}[label=(\alph*)]
\item  For $z=(w,s) \in \hat{Z}\backslash (\partial D \times I)$, from the first equation in \eqref{conds}  we have 
\[\sigma_{\min} (Jf_s(w)) \ge \frac{\tau \|p\|}{2a(n, k, d)}.\] By continuity of the left hand side in both $s$ and $w$, there exists a neighborhood $U_z$ of $z\in Z(F)$ such that 
\begin{equation}
\sigma_{\min} (Jf_t(x)) \ge \frac{\tau \|p\|}{4a(n, k, d)} \quad \forall (x,t)\in U_z.
\label{nozero}
\end{equation} 
\
\\
From classical SVD decomposition, there exist orthogonal matrices $R_1(x,t) \in O(c)$ and $R_2(x,t) \in O(n)$ and a $c\times c$ diagonal matrix $\Sigma(x,t)$ depending continuously on $U_z$ such that \begin{equation}\label{JacobianEq1}Jf_t(x) =R_1(x,t)\tilde{\Sigma}(x,t)R_2(x,t)\end{equation} where $\tilde{\Sigma}(x,t)=(\Sigma(x,t)|0)$ is a $c\times n$ matrix which last $n-c$ columns are zeroes. The equation may be reformulated as $$R_1(x,t)\tilde{\Sigma}(x,t)R_2(x,t)   v(x, t)=\frac{\partial F(x, t)}{\partial t}.$$ Consider then \begin{equation} \label{field}
\tilde{v}(x,t):=
\begin{pmatrix}
\Sigma^{-1}(x,t)R_1^{-1}(x,t)\frac{\partial F(x,t)}{\partial t}\\
0
\end{pmatrix}
\end{equation}
and define on $U_z$ the vector field $v_z(x,t):=R_2^{-1}(x,t)\tilde{v}(x,t)$. It is a well-defined time-dependent vector field on $U_z\subset D^n\times I$ and satisfies \cref{JacobianEq1}. Moreover we have 
\begin{equation}
\begin{aligned}
\n{v_z(x,t)} =\|\tilde{v}(x,t)\| &=\bigg\|\Sigma^{-1}(x,t) R_1^{-1}(x,t) \frac{\partial F(x,t)}{\partial t} \bigg\| \leq\\
 &\le \frac{1}{\sigma_{\min} (Jf_t(x))} \cdot\left\| R_1^{-1}(x,t) \frac{\partial F(x,t)}{\partial t}\right\|=\\
 &=\frac{1}{\sigma_{\min} (Jf_t(x))} \cdot\left\| \frac{\partial F(x,t)}{\partial t}\right\|
 \end{aligned}
\end{equation} that gives the bound \[\n{v_z(x,t)}  \le  \frac{4a(n, k, d)}{\tau \|p\|}   \n{f-p}_{C^1}\quad \forall (x,t) \in U_z.\]
\item Consider now a point $z=(w,s) \in \hat Z \cap (\partial D^n \times I)$. Since $w \in \partial D^n$, by the second equation in \eqref{conds} we have that  $$\frac{\tau \|p\|}{2a(n, k, d)} \leq    \sigma_{\min}\left(\left.J f_s(w)\right|_{w^{\perp}} \right),$$ and in particular there exists a neighborhood  $U_z$ of $z$ in $D^n\times I$ such that  $$\frac{\tau \|p\|}{4a(n, k, d)} \leq    \sigma_{\min}\left(\left.J f_t(x)\right|_{x^{\perp}}\right)\quad \forall (x, t)\in U_z.$$ Thus $\left.J f_t(x)\right|_{x^{\perp}}$ has maximal rank on $U_z$.  We need to solve \begin{equation} \label{JacobianEq2} Jf_t(x) v(x,t)=\frac{\partial F(x,t)}{\partial t}\end{equation} for $v(x,t) \in x^{\perp }\subseteq \R^n$. Since $x^{\perp}$ vary smoothly on $x \in D^n$, there exists $L_x: \R^{n} \to \R^n$ such that $L_x(\R^{n-1})=x^{\perp}$ and it is an isometry in its image. We can thus rewrite the equation as $$ Jf_t(x) L_x   \tilde{v}(x,t)=\frac{\partial F(x,t)}{\partial t} $$ for $\tilde{v}(x,t)\in \R^{n-1}\subseteq \R^n$. Similarly to case (a), SVD decomposition allows to decompose the matrix $Jf_t(x)L_x$ and to formulate \eqref{JacobianEq2} as $$ R_1(x,t)\tilde{\Sigma}(x,t)R_2(x,t) \tilde{v}(x,t)=\frac{\partial F(x,t)}{\partial t} $$ for some orthogonal matrices $R_1(x,t) \in O(c)$ and $R_2(x,t) \in O(n)$ and a $c\times c$ diagonal matrix $\Sigma(x,t)$ depending continuously on $U_z$. Define in $U_z$ the vector field $v_z(x,t):=R_2^{-1}(x,t)\tilde{v}(x,t)$ where $\tilde{v}(x,t)$ is analogous to that in \eqref{field}. This vector field lies in $x^\perp$ for all $(x, t)\in U_z$ and satisfies the equation \eqref{JacobianEq2}. Moreover, since $L_x$ is an isometry for any $x\in D^n$  $$
\begin{aligned}
\|v(x,t)\|=\|\tilde{v}(x,t)\| &\leq \frac{1}{\sigma_{\min}( Jf_t(x)L_x )}\left\|\frac{\partial F}{\partial t}\right\| \\
& \le  \frac{4a(n, k, d)}{\tau \|p\|}   \n{f-p}_{C^1}
\end{aligned}
$$
\item Finally if $z=(x,t) \not\in Z(F)$ we simply set $v_z(x,t)=0.$

\end{enumerate}
\
\\
Consider the open cover $\mathcal{U}$ of $D^n\times I$ obtained by taking the open sets defined by the conditions $(a),(b),(c)$ above by varying $z \in D^n\times I$ and with the extra open set $U_0:=D^n\times I\setminus \hat{Z}$. Given a partition of unity $\{\rho_{U_0}\}\cup \{\rho_{U_z}\}_{z\in \hat{Z}}$ subordinated to $\mathcal{U}$ we define a global vector field on $D^n\times I$ by
 \begin{equation}
 \label{ThomField}
v_t(x):=\sum_{z \in \hat{Z}} \rho_{U_z}(x, t) \cdot v_z(x, t).\end{equation} 
Due to \eqref{cond1} and \eqref{cond2}, the vector field $v_t(x)$ solves \eqref{tang} (since \eqref{tang} is a linear equation), and its flow provides the desired isotopy. We conclude the proof by estimating the distance $\dist_H (Z(p)\cap D^n, Z(f)\cap D^n)$. For any $x \in Z(p)\cap D^n$ $$\sup _{x \in Z(p)\cap D^n} \inf _{y \in Z(f)\cap D^n} \|x-y\| \le \int_0^1\n{v_s(x)}ds \le \frac{4a(n, k, d)}{\tau \|p\|}    \n{f-p}_{C^1}. $$ Conversely, by taking the vector field $-v_t(x)$, for any $y \in Z(f)\cap D^n$ we have $$\sup _ {y \in Z(f)\cap D^n}\inf _{x \in Z(p)\cap D^n} \|x-y\|\le \int_0^{-1}\n{v_s(y)}ds \le \frac{4a(n, k, d)}{\tau \|p\|}    \n{f-p}_{C^1}. $$ that gives the desired bound.
\end{proof}

As a corollary, we now give the proof of \cref{thm:thom}.
\begin{proof}[Proof of \cref{thm:thom}]Observe first that, since all norms in finite dimension are equivalent, there exists $c(n,k,d)>0$ such that for all $q\in \Pol$
\[\|q\|_{C^1}\leq c(n,k,d)\|q\|.\]
Therefore, applying \cref{propo:thom} with the choice  $\tau:=\mathrm{dist}(\Sigma, p)/\|p\|$ and with $f=q\in \Pol$ such that $\|q-p\|\leq \mathrm{dist}(p, \Sigma)/(2c(n,k,d))$ (this quantifies the ``sufficiently small'' in the statement), we get
\[ \mathrm{dist}_H\left(\kappa(p), \kappa(q)\right)\leq \frac{L}{\mathrm{dist}(p, \Sigma(d))}\cdot \|p-q\|,\]
where $L:=4 a(n,k,d)c(n,k,d).$
\end{proof}

\subsection{Hausdorff geometry of the space of complete intersections}We discuss in this section the Hausdorff geometry of the space of complete intersections and its closure.

The goal of this section is to prove the following theorem.
\begin{theorem}\label{thm:tau}For every $d,n,k\in \mathbb{N}$ there exists $\alpha, \beta, \epsilon_0>0$ such that for every $0<\epsilon<\epsilon_0$ and for every $Z\in \HH(n,k,d)$ with $k=\dim(Z)=n-c$, there exists $p_\epsilon=(p_1, \ldots, p_c)\in P(2d)$ such that 
\[\mathrm{dist}_{H}(Z, Z(p_\epsilon)\cap D^n)\leq \epsilon\quad \textrm{and}\quad \mathrm{dist}(p_\epsilon,\Sigma(2d))\geq \|p_\epsilon\|\alpha \epsilon^\beta.\]
\end{theorem}
Notice that, since for every $\lambda>0$ the zero sets of $p$ and $\lambda p$ are the same and $\Sigma(2d)$ is a cone (meaning that, if $q\in \Sigma(2d)$ and $\lambda>0$ then $\lambda q\in \Sigma(2d)$), in the above statement we can assume that $\|p_\epsilon\|=1$. 

We note that approximation of arbitrarily algebraic sets with complete intersections follows already by \cite{BasuLerario} (see \cref{thm:BasuLerario1} below --  even if technically here we need the stronger result \cref{thm:BasuLerario}). What is important here, as a consequence of \cref{thm:tau},
is that we can control how far we can stay away from the set
of singular algebraic sets while being $\epsilon$--close to the set we are approximating.

Before giving the proof of the theorem, let us establish some preliminary results. First, we record the elementary fact that the Hausdorff distance between elements in a semialgebraic family is a semialgebraic function.
\begin{lemma}\label{lemma:distH}Let $A\subset P\times \R^n$ be a semialgebraic set such that for every $p\in P$ the set $A_p:=\{x\in \R^n\,|\, (p, x)\in A\}$ is compact. Then the function
\[(p_1, p_2)\mapsto \mathrm{dist}_H(A_{p_1}, A_{p_2})\]
is semialgebraic.
\end{lemma}
\begin{proof}This follows immediately from the definition \eqref{def:Hdist}, since the involved functions are semialgebraic.
\end{proof}

Using the notation from \cref{lemma:distH}, we consider now the semialgebraic set
\[A:=\bigg\{(p,x)\in P(d) \times D^n\,\bigg|\,  \|p\|=1,\, p(x)=0, \bigg\},\]
so that the projection on the first factor of $A$ equals $P(d)\cap \{\|p\|=1\}$.
With this notation, note that the map $\kappa:P(d)\to \mathcal{K}$ introduced above, 
equals
\[\kappa(p)=A_p:=Z(p)\cap D^n.\]
Recall also the definition of the set $T=T(d):=P(d)\setminus \Sigma(d)$.

\begin{lemma}\label{lemma:continuity}The restriction of $\kappa$ to $T(d)$ is continuous. On the other hand, $\kappa$ is not continuous on the set $P(d)$.
\end{lemma}
\begin{proof}The first statement follows from \cref{thm:thom}.

To see that, in general, $\kappa$ is not continuous on $P(d)$ consider, for instance the sequence of polynomials $p_j(x):=(x-\frac{1}{2})(x^2+\frac{1}{j})$. Then
\[\lim_{j\to \infty}p_j(x)=p_\infty(x):=x\left(x-\frac{1}{2}\right).\]
On the other hand,
\[\lim_{j\to \infty}(Z(p_j)\cap D^1)=\left\{\frac{1}{2}\right\}\neq Z(p_\infty)\cap D^1=\left\{0, \frac{1}{2}\right\}.\]
(It is easy to construct similar examples in several variables.
For instance, if $p\in \Sigma$ is such that there is a unique $x\in D^n$ where $p(x)=\frac{\partial p}{\partial x_1}=\cdots=\frac{\partial p}{\partial x_n}=0$ and at this point the Hessian matrix of $p$ is positive (or negative) definite, then $\kappa$ is not continuous at $p$.)
\end{proof}
In the next  proposition we will use the following elementary lemma.
\begin{lemma}\label{lemma:function}Let $f:(0, \infty)\to [0, \infty)$ be a semialgebraic function such that
 \[\lim_{r\to 0}f(r)=0.\] Then there exist $\alpha_0, \beta_0, s_0>0$ such that
\begin{enumerate}
\item  $f|_{(0, s_0)}:(0, s_0)\to (0, f(s_0))$ is invertible;
\item $f(s)\leq \alpha_0s^{\beta_0}$ for all $0<s<s_0.$
\end{enumerate}
\end{lemma}
\begin{proof}This follows immediately from the fact on some $(0, r_1)$ the function $f$ must be $C^1$ (\cite[Chapter 7, Theorem 3.2]{Vandendries}) and that germs of semialgebraic continuous functions on a positive interval are algebraic Puiseux series (\cite[Theorem 3.14]{BasuPollackRoy}).
\end{proof}
\begin{proposition}\label{propo:universal}The function $h:(0, \infty)\to \R$ defined by
\[h(\delta):=\sup_{C\in\overline{\kappa( T)}^{\mathrm{H}}}\left(\inf_{p\in T,\,\mathrm{dist}(p, \Sigma)\geq \delta\|p\|} \mathrm{dist}_{H}(C, \kappa(p))\right),\]
where the supremum is taken over the closure in the Hausdorff topology of  $\kappa(T)$, is semialgebraic. Moreover,
\[\label{eq:limitdelta}\lim_{\delta\to 0}h(\delta)=0.\]
\end{proposition}
\begin{proof}Observe first that, since the supremum is taken over the closure of $\kappa(T)$, we can alternatively write the function $h$ as
\begin{align}h(\delta)&=\sup_{C\in\kappa( T)}\left(\inf_{p\in T,\,\mathrm{dist}(p, \Sigma)\geq \delta\|p\|} \mathrm{dist}_{H}(C, \kappa(p))\right)\\
&=\sup_{q\in T}\left(\inf_{p\in T,\,\mathrm{dist}(p, \Sigma)\geq \delta\|p\|} \mathrm{dist}_{H}(\kappa(q), \kappa(p))\right).\end{align}

Notice that, because $\kappa$ is scale--invariant, i.e. $\kappa(\lambda p)=\kappa(p)$ for every $\lambda\neq 0$, we may restrict to work with the set 
\[\tilde{T}:=T\cap \{\|p\|=1\},\]
so that
\[h(\delta)=\sup_{q\in \tilde{T}}\left(\inf_{p\in \tilde{T},\,\mathrm{dist}(p, \Sigma)\geq \delta} \mathrm{dist}_{H}(\kappa(q), \kappa(p))\right).\]

Now, the function $(p,q)\mapsto \mathrm{dist}_{H}(\kappa(p), \kappa(q))$ is semialgebraic by \cref{lemma:distH}, therefore so is $h$. In particular, the limit $L:=\lim_{\delta\to 0}h(\delta)$
exists.  

Assume by contradiction that $L>0$. Then there exist sequences $\{\delta_k\}_{k\in \N}\subset (0, \infty)$ and  $\{q_k\}_{k\in \N}\subset T$ satisfying  $\lim_{k\to \infty}\delta_k\to 0$ and 
\[\label{eq:L}\mathrm{dist}(p, \Sigma)\geq \delta_k\implies \mathrm{dist}_{H}(\kappa(q_k), \kappa(p))\geq \frac{L}{2}.\]

Since $\mathcal{K}$ is compact, the sequence $\{\kappa(q_k)\}_{k\in \N}$ (up to subsequences) converges to some $\overline{C}\in \mathcal{K}$. Therefore, for every $k\in \N$ we must have 
\[\label{eq:contr}\mathrm{dist}(p, \Sigma)\geq \delta_k\implies \mathrm{dist}_H(\overline{C}, \kappa(p))\geq \frac{L}{4}.\] Otherwise, by triangle inequality, 
\[\lim_{k\to \infty}\mathrm{dist}_H(\kappa(q_k), \kappa(p))\leq \lim_{k\to \infty}\left(\mathrm{dist}_H(\kappa(q_k), \overline{C})+\mathrm{dist}_H(\overline{C}, \kappa(p))\right)\leq \frac{L}{4},\]
which goes against \eqref{eq:L}.

By \cite[Corollary 2]{definability}, there exists a semialgebraic arc $\gamma:(0,1)\to T$ such that 
\[\lim_{s\to 0}\mathrm{dist}_{H}(\overline{C}, \kappa(\gamma(s)))=0.\]
Since $\gamma$ is semialgebraic and bounded, the limit $\lim_{s\to0}\gamma(s)$ exists. This limit cannot be in $T$, otherwise, since $\kappa $ is continuous on $T$ (by \cref{lemma:continuity}), we would have $\overline{C}=\kappa(\overline{p})$ for some $\overline{p}\in T$. Such $\overline{p}$ lies definitely in one of the sets $\{\mathrm{dist}(\Sigma, \cdot)\geq \delta_k\}$, since $\delta_k\to 0$, which goes against \eqref{eq:contr}.

 Therefore the limit of the arc $\gamma$ is in $\Sigma$, which implies 
\[\lim_{s\to 0}\mathrm{dist}(\Sigma, \gamma(s))=0.\]
The function $f(s):=\mathrm{dist}(\Sigma, \gamma(s))$ is semialgebraic, therefore it is invertible when restricted to some interval $(0, s_0)$ by \cref{lemma:function}. In particular for every $k\in \N$ such that $\delta_k\leq f(s_0)$ there exists $r_k\in (0, s_0)$ such that $\mathrm{dist}(\Sigma, \gamma(r_k))= \delta_k.$ This goes now against \eqref{eq:contr} and gives the desired contradiction.
\end{proof}

The last key ingredient for the proof of \cref{thm:tau} is the following result from \cite{BasuLerario}.
\begin{theorem}[{\cite[Theorem 2.10]{BasuLerario}}]\label{thm:BasuLerario1}Let $Z\subset \R^n$ be an algebraic set defined by polynomials of degree at most $d$, of dimension $k=n-c$ and such that $Z\cap D^n\neq \emptyset$. Then, there exists a one parameter family of smooth complete intersections $\{Z_t\}_{t>0}$ defined by $c$ polynomials of degree at most $2d$ and such that
\[\lim_{t\to 0}\mathrm{dist}_H(Z_t\cap D^n, Z\cap D^n)=0.\]
\end{theorem}
We will actually need the following stronger version of the previous result, stating that the one parameter family of smooth complete intersections $\{Z_t\}_{t>0}$ from \cref{thm:BasuLerario1} can be chosen so that, for all $t>0$ small enough, the intersection $Z_t\cap \partial D^n$ is transversal (\cref{thm:BasuLerario1} only guarantees that $Z_t$ is a smooth complete intersection, without mentioning its behavior on the boundary of the disk). 

\begin{theorem}\label{thm:BasuLerario}For every $p\in P(d)$ with $\dim(Z(p)\cap D^n)=k$ there exists a one parameter family $p_\epsilon \in P(2d)\setminus \Sigma(2d)$ such that
\[\lim_{\epsilon\to 0}\mathrm{dist}_H(Z(p_\epsilon)\cap D^n, Z(p)\cap D^n)=0.\]
\end{theorem}
\begin{proof}The proof goes exactly as the proof of \cref{thm:BasuLerario1}, with the following modification. Using the same notation from that proof, 
the polynomial 
$G\in \R[x]_{2d}$, which in \cref{thm:BasuLerario1} is chosen such that such that $G\geq 0$
and such that for every $0 \leq k \leq n$ the set $\mathrm{Cr}^h_k(G)$ defines a smooth complete intersection in $\C\mathrm{P}^n$, is chosen so that the additional condition that the zero sets of $\mathrm{Cr}^h_k(G)$ and $z_1^2+\cdots + z_n^2-z_0^2$  are transversal is satisfied. This is possible by genericity of both choices. More precisely: (1) the set of polynomials $G\in \R[x]_{2d}$ such that $G\geq 0$ is a full--dimensional cone $C_{2d}$ in $\R[x]_{2d}$; (2) the set $U_1$ of polynomials $G\in \R[x]_{2d}$ such that $\mathrm{Cr}^h_k(G)$ defines a smooth complete intersection in $\C\mathrm{P}^n$ contains an open and dense semialgebraic set in $\R[x]_{2d}$; (3) the set $U_2$ of polynomials $G\in \R[x]_{2d}$ such that  $\mathrm{Cr}^h_k(G)$ and $z_1^2+\cdots + z_n^2-z_0^2$  are transversal contains an open and dense semialgebraic set in $\R[x]_{2d}$. Therefore, $C_{2d}\cap U_1\cap U_2$ is nonempty and we choose an element  $G$ in it. With this choice, the one parameter family of complete intersections from \cref{thm:BasuLerario1} is also transversal to $\partial D=Z(z_1^2+\cdots + z_n^2-z_0^2)\cap \R^n$.\end{proof}

We are ready now for the proof of \cref{thm:tau}. 
\begin{proof}[Proof of \cref{thm:tau}]Consider the function $h$ from \cref{propo:universal} for the case $T=T(2d)\subset P(2d).$ Since $h$ is semialgebraic and \eqref{eq:limitdelta} holds, by \cref{lemma:function} there exist $\delta_0, \alpha_0, \beta_0>0$ such that
for all $0<\delta<\delta_0$ one has $h(\delta)\leq \alpha_0\delta^{\beta_0}$. Set 
\[\epsilon_0:=\alpha_0\delta_0^{\beta_0}, \quad \alpha:=\frac{1}{\alpha_0}, \quad \beta:=\frac{1}{\beta_0}.\]
Then, for all $\epsilon<\epsilon_0$ we have $h(\alpha \epsilon^{\beta})\leq \epsilon.$

Let now $Z\subset \R^n$ be an algebraic set defined by polynomials of degree at most $d$ and of dimension $n-c$. It follows now from \cref{thm:BasuLerario} that the set $C:=Z\cap D^n$ belongs to Hausdorff closure of the image of $\kappa:T(2d)=P(2d)\setminus \Sigma(2d)\to \mathcal{K}.$ Therefore, denoting by $\tilde{T}(2d):=T(2d)\cap \{\|p\|=1\},$ by \cref{propo:universal},
\[\label{eq:infimum}\inf_{p\in \tilde{T}(2d),\,\mathrm{dist}(p, \Sigma(2d))\geq \alpha \epsilon^\beta} \mathrm{dist}_{H}(C, \kappa(p))\leq h(\alpha \epsilon^\beta)\leq \epsilon.\]
Observe now that, for $\epsilon>0$, the set $\{p\in \tilde{T}(2d),\,\mathrm{dist}(p, \Sigma(2d))\geq \alpha \epsilon^\beta\}$ is compact in $P(2d)$:  the condition $\|p\|=1$ is a closed condition and implies the set is contained in the unit sphere in the space of polynomials; the condition $Z(p)\cap D^n\neq \emptyset$ is a closed condition (since $D^n$ is compact) and so is $\mathrm{dist}(p, \Sigma)\geq \alpha \epsilon^\beta$. Therefore the infimum in \eqref{eq:infimum} is actually attained at some $p_\epsilon$.
\end{proof}

\subsection{A quantitative precompactness result}For a compact set $Z\subset D^n$ and $\epsilon>0$, we denote by $B_{H}(Z, \epsilon)\subset \mathcal{K}(D^n)$ the $\epsilon$--Hausdorff ball centered at $Z$, i.e. set of compact subsets of $D^n$ with Hausdorff distance at most $\epsilon$ from $Z$. 
The following theorem gives a quantitative precompactness result for the space $\HH(n,k,d)\subset \mathcal{K}(D^n)$ in the Hausdorff topology. More precisely, since $\mathcal{K}(D^n)$ is compact, so is the closure of $\HH(n,k,d)$ in it and, since $\mathcal{K}(D^n)$ is a metric space, we know that $\HH(n,k,d)$ is totally bounded. Here, for every $\epsilon>0$, we give a bound on the number of $\epsilon$--balls needed to cover it.

\begin{theorem}\label{coro:net}For every $n,k,d$ there exists $\epsilon_0>0$ and $a_1, a_2>0$ such that for every $0<\epsilon<\epsilon_0$ there are polynomials $q_1, \ldots, q_{\nu(\epsilon)}\in P(2d)\setminus \Sigma(2d)$ such that 
 \[\HH(n,k,d)\subset \bigcup_{i=1}^{\nu(\epsilon)}B_H(Z(q_i)\cap D^n, \epsilon)\quad \textrm{and}\quad \nu(\epsilon)\leq a_1\epsilon^{-a_2}.\] 
\end{theorem}
\begin{proof}Let $N:=(n-k){{n+2d}\choose{2d}}$ denote the dimension of the space $P(2d)$ and denote by $B^N$ its unit Bombieri--Weyl ball. Let also $\alpha, \beta>0$ be given by \cref{thm:tau} and $a(n)>0$ be given by \cref{thm:thom} and define
\[ r(\epsilon):= a(n)\cdot\alpha\cdot \left(\frac{\epsilon}{2}\right)^{\beta+1}.\] 

By \cref{thm:tau}, for every element $Z\in \HH(n,k,d)$ there is $p\in P(2d)$, with $\|p\|\leq 1$ and such that 
\[\mathrm{dist}_H(Z(p)\cap D^n, Z)\leq \frac{\epsilon}{2}\quad \textrm{and}\quad \mathrm{dist}(p, \Sigma(2d))\geq \alpha \left(\frac{\epsilon}{2}\right)^\beta\|p\|.\]
(Since $Z(p)=Z(\lambda p)$, for every $\lambda\neq 0$, we may assume that $\|p\|\leq 1$.)

By \cref{thm:thom}, for every $q\in P(2d)$ with $\mathrm{dist}(p, \Sigma(2d))\geq \alpha \left(\frac{\epsilon}{2}\right)^\beta\|p\|$  
\[\label{eq:muimplies}\kappa\left(B_{C^1}(q, r(\epsilon))\right)\subset B_H\left(Z(q)\cap D^n, \tfrac{\epsilon}{2}\right).\] 

Consider now an $\left(\tfrac{r(\epsilon)}{4}\right)$--net in $B^N$ for the $C^1$ norm and denote by $\{q_1, \ldots, q_{\nu(\epsilon)}\}$ the elements from the net that are at distance at least $\alpha (\epsilon/2)^\beta$ from $\Sigma(2d)$. Then, denoting by
\[ U:=\left\{p\in P(2d)\,\bigg|\,\mathrm{dist}(p, \Sigma(2d))\geq \alpha \left(\frac{\epsilon}{2}\right)^\beta\right\}\cap B^{N}\subset B^N,\]
we see that 
\[ U\subset \bigcup_{i=1}^{\nu(\epsilon)}B_{C^1}(q_i, r(\epsilon)).\]
Consequently, every $Z\in \HH(n,k,d)$ will be at Hausdorff distance at most $\epsilon$ from one of the centers $\{q_1, \ldots, q_{\nu(\epsilon)}\}$ of these balls and \eqref{eq:muimplies} implies
\[\HH(n,k,d)\subset \bigcup_{i=1}^{\nu(\epsilon)}B_H(Z(q_i)\cap D^n, \epsilon).\]

It remains to estimate the cardinality $\nu(\epsilon)$. By construction, this cardinality is bounded by the cardinality of a $\left(\tfrac{r(\epsilon)}{4}\right)$--net in $B^N$ for the $C^1$ norm. Since all norms in finite dimension are equivalent, there exists $C=C(n,k,d)$ such that the cardinality of such net can be bounded by the cardinality of a $\left(\frac{Cr(\epsilon)}{4}\right)$--net in $B^N$ for the Bombieri--Weyl norm, and this can be bounded by
\[\nu(\epsilon)\leq (Cr(\epsilon))^{-N}=a_1\epsilon^{-a_2}.\]
\end{proof}
\section{Consequences for manifold learning}As a corollary of \cref{thm:thom} and \cref{thm:tau}, we prove now \cref{thm:Nvariation} (\cref{thm:Nvariationintro} from the Introduction), a result that serves as an alternative to  \cite[Theorem 4]{Narayanan}. As we already observed, compared to \cite[Theorem 4]{Narayanan}, our \cref{thm:Nvariation} has a better dependence on $\epsilon$ (but no explicit control on the implied constants).

We begin by establishing some preliminary result.
First comes a direct consequence of \cref{thm:BasuLerario}.

\begin{proposition}Let $\mu\in \mathcal{P}(D^n)$ be a Borel measure. For every $n,k,d\in \N$ and $\epsilon>0$ there exists $Z_\epsilon\in \HH(n,k,2d)$ that is a smooth complete intersection and such that
\[\left|\RR(Z_\epsilon, \mu)-\inf_{Z\in \HH(n,k,d)}\RR(Z, \mu)\right|\leq \epsilon.\]
Similarly, for every $n,k,d, m\in \N$, $\epsilon>0$ and $\underline{x}_m=(x^1, \ldots, x^m)\in (D^n)^m$, denoting by $\mu_m$ the discrete measure
\[\mu_m:=\frac{1}{m}\sum_{i=1}^m\delta_{x^i},\]
there exists $\widehat{Z}_{\epsilon}\in \HH(n,k,2d)$ that is a smooth complete intersection and such that
\[\left|\RR(\widehat{Z}_{\epsilon}, \mu_m)-\inf_{Z\in \HH(n,k,d)}\RR(Z, \mu_m)\right|\leq \epsilon.\]
\begin{proof}Both statements follows from the continuity of $\RR(\cdot, \mu)$ with respect to the Hausdorff distance and \cref{thm:BasuLerario}.
\end{proof}
\end{proposition}
Next is the following lemma, which is essentially Hoeffding's bound.
 \begin{lemma}\label{lemma:balls}Let $\mu\in \mathcal{P}(D^n)$ and $\{x^i\}_{i\in \N}$ be a sequence of i.i.d. random variables sampled according to $\mu$ and, for every $m\in \N$, denote by $\mu_m$ the corresponding empirical random measure. Let also $Z_0\subset D^n$ be a compact set. Then, for every $\epsilon>0$ and $m\in \N$,
 \[\mathbb{P}\left\{\sup_{\{Z\,|\,\mathrm{dist}_H(Z, Z_0)\leq \frac{\epsilon}{16}\}}\left|\RR(Z, \mu_m)-\RR(Z, \mu)\right|>\epsilon\right\}\leq 2e^{-\frac{m\epsilon^2}{32}}. 
\]
 \end{lemma}
 \begin{proof}
 Observe first that for all closed sets $Z, Z_0\subseteq D^n$ such that $\mathrm{dist}_H(Z, Z_0)\leq \frac{\epsilon}{16}$  and for all $x\in D^n$, since $\mathrm{diam}(D)=2$, we have
\[\left|\mathrm{dist}(x, Z)^2-\mathrm{dist}(x, Z_0)^2\right|=\left|\mathrm{dist}(x, Z)-\mathrm{dist}(x, Z_0)\right|\cdot \left |\mathrm{dist}(x, Z)+\mathrm{dist}(x, Z_0)\right|\leq \frac{\epsilon}{4}.\]
In particular,  for all closed sets $Z\subseteq D^n$ such that $\mathrm{dist}_H(Z, Z_0)\leq \frac{\epsilon}{16}$ and for all $\underline{x}_m\in (D^n)^m$ 
\[\left|\RR(Z, \mu_m)-\RR(Z_0, \mu_m)\right|\leq \frac{\epsilon}{4}\quad \textrm{and}\quad \left|\RR(Z_0, \mu)-\RR(Z, \mu)\right|\leq \frac{\epsilon}{4},\]
and, therefore, 
\begin{align}\left|\RR(Z, \mu_m)-\RR(Z, \mu)\right|\leq&\left| \RR(Z, \mu_m)-\RR(Z_0, \mu_m)\right|+\\
&+\left|\RR(Z_0, \mu_m)-\RR(Z_0, \mu)\right|+\left|\RR(Z_0, \mu)-\RR(Z, \mu)\right| \\
\leq& \frac{\epsilon}{4}+\left|\RR(Z_0, \mu_m)-\RR(Z_0, \mu)\right|+\frac{\epsilon}{4}\\
\leq&\label{eq:ineqZ0}\left|\RR(Z_0, \mu_m)-\RR(Z_0, \mu)\right|+\frac{\epsilon}{2}.
\end{align}
For a given closed set $Z\subset D^n$, consider now the event
\[A_m(Z, \epsilon):=\left\{\underline{x}_m\in (D^n)^m\,\bigg|\, \left|\RR(Z, \mu_m)-\RR(Z, \mu)\right|>\epsilon\right\}.\]
The inequality \eqref{eq:ineqZ0} implies that
\[\label{eq:impliH}\mathrm{dist}_H(Z, Z_0)\leq \frac{\epsilon}{16}\implies A_m(Z, \epsilon)\subseteq A_m(Z_0, \tfrac{\epsilon}{2}).\]
By Hoeffding's inequality\footnote{The statement of Hoeffding's inequality is the following. Let $\{X_j\}_{j \in \N}$ be a sequence of i.i.d. random variables with $a\leq X_j\leq b.$ Then for all $\lambda>0$
\[\mathbb{P}\left\{\left|\frac{1}{m}\sum_{j=1}^m X_j-\mathbb{E}X_1\right|>\lambda\right\}\leq 2e^{-\frac{2m\lambda^2}{(b-a)^2}}.\]
We apply this inequality with the choices $X_j:=\mathrm{dist}(Z_0, x_j)^2$, so that $0\leq X_j\leq 4$, and $\lambda=\frac{\epsilon}{2}$, giving \eqref{eq:hoeff}.}, for every $\epsilon>0$ and $m\in \N$ we have:
\[\label{eq:hoeff}\mathbb{P}(A_m(Z_0, \tfrac{\epsilon}{2}))=\mathbb{P}\left\{\left|\RR(Z_0, \mu_m)-\RR(Z_0, \mu)\right|>\frac{\epsilon}{2}\right\}\leq 2e^{-\frac{m\epsilon^2}{32}}.\]
Together with \eqref{eq:impliH}, this implies that for all $\epsilon>0$ and $m\in\N$ we have:
\begin{align}\mathbb{P}\left\{\sup_{\{Z\,|\,\mathrm{dist}_H(Z, Z_0)\leq \frac{\epsilon}{16}\}}\left|\RR(Z, \mu_m)-\RR(Z, \mu)\right|>\epsilon\right\}&\leq\mathbb{P}\left(\bigcup_{\{Z\,|\,\mathrm{dist}_H(Z, Z_0)\leq \epsilon\}}A_m(Z, \epsilon)\right)\\
&\leq \mathbb{P}(A_m(Z_0, \tfrac{\epsilon}{2}))\leq2e^{-\frac{m\epsilon^2}{32}}. 
\end{align}
 \end{proof}
We re now ready for the proof of \cref{thm:Nvariation}. The main idea for the proof is to pair the concentration bound from \cref{lemma:balls} with an $\epsilon$--covering argument for the space $\HH(n,k,d)$ with respect to the Hausdorff metric.  Since $\HH(n,k,d)$ is precompact in $\mathcal{K}(D^n)$, such a covering always exists, but the key point is controlling  its cardinality $\nu(\epsilon)$ in an explicit way: roughly speaking this should not grow exponentially fast in $\epsilon$, which is the content of \cref{coro:net}.
\begin{theorem}\label{thm:Nvariation}For every $n,k,d\in \N$ there exist $c_0,c_1, c_2, c_3>0$ such that the following statement is true. Let $\mu$ be a Borel probability measure on $D^n$ and for all $m\in \N$ denote by $\mu_m$ the corresponding random empirical measure. For all $0<\epsilon< c_0$ and $0<\delta<1$, 
if
\[\label{eq:boundnkd}m\geq \frac{c_3}{\epsilon^2}\log\left(\frac{c_1 \epsilon^{-c_2}}{\delta}\right),\]
then
\[\mathbb{P}\left\{\sup_{Z\in \HH(n,k,d)}\left|\RR(Z, \mu_m)-\RR(Z,\mu)\right|> \epsilon\right\}<\delta.\]
\end{theorem}\begin{proof}By \cref{coro:net} for $0<\epsilon<\epsilon_0/16=:c_0$ we can find compact sets $Z_1, \ldots, Z_{u(\epsilon)}\subset D^n$, with $Z_i:=Z(q_i)\cap D^n$ for some $q_i\in P(2d)$ nonsingular given by \cref{coro:net}, such that we can cover our hypothesis class $\HH(n,k,d)$ with the $\tfrac{\epsilon}{16}$--Hausdorff balls centered at the $Z_i$, i.e. we can write
\[\HH(n,k,d)\subseteq \bigcup_{i=1}^{u(\epsilon)}B_{H}(Z_i, \tfrac{\epsilon}{16}).\]
The cardinality of this covering can be bounded by
\[u(\epsilon)= \nu\left(\frac{\epsilon}{16}\right)\leq a_1\left(\frac{\epsilon}{16}\right)^{a_2}. \]
For every $i=1, \ldots, u(\epsilon)$, denote by
\[E_i(m,\epsilon):=\left\{\sup_{\{Z\,|\,\mathrm{dist}_H(Z, Z_i)\leq \frac{\epsilon}{16}\}}\left|\RR(Z, \mu_m)-\RR(Z, \mu)\right|>\epsilon\right\}\]
Then we can estimate 
\begin{align}\mathbb{P}\left\{\sup_{Z\in \HH(n,k,d)}\left|\RR(Z, \mu)-\RR(Z,\mu_m)\right|>\epsilon\right\}&\leq\mathbb{P}\left(\bigcup_{i=1}^{u(\epsilon)}E_i(m, \epsilon)\right)\\
&\leq u(\epsilon)\sup_{i=1, \ldots, u(\epsilon)}\mathbb{P}(E_i(m, \epsilon))\\
&\leq u(\epsilon)2e^{-\frac{m\epsilon^2}{32}}\quad (\textrm{by \cref{lemma:balls}})\\
&\leq a_1\left(\frac{\epsilon}{16}\right)^{-a_2}2e^{-\frac{m\epsilon^2}{32}}.\end{align}
The above probability is now smaller than $0<\delta<1$ if
 \[m\geq \frac{32}{\epsilon^2}\log\left(\frac{2a_1}{\delta}\left(\frac{\epsilon}{16}\right)^{-a_2}\right)=:\frac{c_3}{\epsilon^2}\log\left(\frac{c_1 \epsilon^{-c_2}}{\delta}\right).\]
\end{proof}

\bibliographystyle{alpha}
\bibliography{bibliototale}

\newcommand{\etalchar}[1]{$^{#1}$}
\begin{thebibliography}{DHO{\etalchar{+}}16}

\bibitem[BCR98]{BCR}
Jacek Bochnak, Michel Coste, and Marie-Francoise Roy.
\newblock {\em Real algebraic geometry}, volume~36 of {\em Ergebnisse der Mathematik und ihrer Grenzgebiete (3) [Results in Mathematics and Related Areas (3)]}.
\newblock Springer-Verlag, Berlin, 1998.
\newblock Translated from the 1987 French original, Revised by the authors.

\bibitem[BKS24]{metricalgebraicgeometry}
Paul Breiding, Kathl\'en Kohn, and Bernd Sturmfels.
\newblock {\em Metric algebraic geometry}, volume~53 of {\em Oberwolfach Seminars}.
\newblock Birkh\"auser/Springer, Cham, [2024] \copyright 2024.

\bibitem[BL23]{BasuLerario}
Saugata Basu and Antonio Lerario.
\newblock Hausdorff approximations and volume of tubes of singular algebraic sets.
\newblock {\em Math. Ann.}, 387(1-2):79--109, 2023.

\bibitem[BPR06]{BasuPollackRoy}
Saugata Basu, Richard Pollack, and Marie-Francoise Roy.
\newblock {\em Algorithms in real algebraic geometry}, volume~10 of {\em Algorithms and Computation in Mathematics}.
\newblock Springer-Verlag, Berlin, second edition, 2006.

\bibitem[DHO{\etalchar{+}}16]{EDD}
Jan Draisma, Emil Horobet, Giorgio Ottaviani, Bernd Sturmfels, and Rekha~R. Thomas.
\newblock The {E}uclidean distance degree of an algebraic variety.
\newblock {\em Found. Comput. Math.}, 16(1):99--149, 2016.

\bibitem[Fed59]{Federer}
Herbert Federer.
\newblock Curvature measures.
\newblock {\em Trans. Amer. Math. Soc.}, 93:418--491, 1959.

\bibitem[FMN16]{Fefferman}
Charles Fefferman, Sanjoy Mitter, and Hariharan Narayanan.
\newblock Testing the manifold hypothesis.
\newblock {\em J. Amer. Math. Soc.}, 29(4):983--1049, 2016.

\bibitem[KCPV14]{definability}
Beata Kocel-Cynk, Wieslaw Pawlucki, and Anna Valette.
\newblock A short geometric proof that {H}ausdorff limits are definable in any o-minimal structure.
\newblock {\em Adv. Geom.}, 14(1):49--58, 2014.

\bibitem[Ler23]{lecturenotesantonio}
Antonio Lerario.
\newblock Lectures on metric algebraic geometry.
\newblock {\em Lecture notes, available: \url{https://drive.google.com/file/d/1A6UzYuv1OjucRscwZOQ4mDakKSfk_c77/view}}, 2023.

\bibitem[Mil64]{Milnor}
J.~Milnor.
\newblock On the {{Betti}} numbers of real varieties.
\newblock {\em Proceedings of the American Mathematical Society}, 15(2):275--280, 1964.

\bibitem[Nar12]{Narayanan}
Hariharan Narayanan.
\newblock Sample complexity in manifold learning.
\newblock In {\em Manifold learning theory and applications}, pages 73--93. CRC Press, Boca Raton, FL, 2012.

\bibitem[Raf14]{Raffalli}
Christophe Raffalli.
\newblock Distance to the discriminant, 2014.

\bibitem[Ric68]{Richardson}
Daniel Richardson.
\newblock Some undecidable problems involving elementary functions of a real variable.
\newblock {\em J. Symbolic Logic}, 33:514--520, 1968.

\bibitem[Vap00]{VladimirVapnik}
Vladimir Vapnik.
\newblock {\em The Nature of Statistical Learning Theory}.
\newblock Springer, 2nd edition, 2000.

\bibitem[vdD98]{Vandendries}
Lou van~den Dries.
\newblock {\em Tame topology and o-minimal structures}, volume 248 of {\em London Mathematical Society Lecture Note Series}.
\newblock Cambridge University Press, Cambridge, 1998.

\bibitem[YC04]{ComteYomdin}
Yosef Yomdin and Georges Comte.
\newblock {\em Tame geometry with application in smooth analysis}, volume 1834 of {\em Lecture Notes in Mathematics}.
\newblock Springer-Verlag, Berlin, 2004.

\bibitem[Zha23]{TongZhang}
Tong Zhang.
\newblock {\em Mathematical Analysis of Machine Learning Algorithms}.
\newblock Cambridge University Press, 2023.

\bibitem[ZK23]{KileelZhang}
Yifan Zhang and Joe Kileel.
\newblock Covering number of real algebraic varieties and beyond: Improved bounds and applications, 2023.

\end{thebibliography}

\end{document}